\documentclass[dvips,???,preprint]{imsart}
\RequirePackage[OT1]{fontenc}
\RequirePackage{amsthm,amsmath}
\RequirePackage[numbers]{natbib}

\usepackage{graphicx,amsmath,amsfonts,amscd,amssymb,bm,epsfig,epsf,url,
color,algorithm,algorithmic, mathrsfs}
\usepackage[small,bf]{caption}
\startlocaldefs
\newtheorem{theorem}{Theorem}[section]
\newtheorem{lemma}[theorem]{Lemma}

\newtheorem{proposition}[theorem]{Proposition}

\newtheorem{remark}[theorem]{Remark}

\newcommand{\diag}{\operatorname{diag}}

\numberwithin{equation}{section}
\endlocaldefs

\begin{document}

\begin{frontmatter}

\title{Eigenvalues of Deformed Random Matrices}
\runtitle{Eigenvalues of Deformed Random Matrices}

\begin{aug}
\author{\fnms{Minyu Peng}}
\ead[label=a1]{mpeng@stanford.edu}
\address{
Stanford University\\
Department of Mathematics\\
Building 380, Sloan Hall\\
Stanford, CA, 94305\\
\printead{a1}}
\affiliation{Stanford University}
\runauthor{Minyu Peng}
\end{aug}

\begin{abstract}
We present a novel approach to prove non-asymptotic deviation bounds
for extreme eigenvalues of deformed random matrices. This approach
applies to many deformed Gaussian matrix models; two such models are
studied in detail: the deformed GOE and the spiked population model.
\end{abstract}



\end{frontmatter}

\maketitle
\section{Introduction}

The aim of this paper is to study extreme eigenvalues of matrices
from deformed random matrix ensembles. We will consider both the
deformed GOE and the spiked population model.

\subsection{Models and some known results}

The Gaussian Orthogonal Ensemble, or GOE for short, is probably the
most widely studied model in random matrix theory. The deformed GOE
is a finite rank perturbation of the Gaussian Orthogonal Ensemble.
More precisely, let $G\in GOE(n,\frac{\sigma^2}{n})$ and $P$ be a
real symmetric matrix, we want to study the extreme eigenvalues of
$A=P+G$.

When the dimension goes to infinity, the asymptotic properties of
the largest eigenvalues of matrices from the  deformed GOE has been
studied by various authors, where the a.e. limit, CLT and large
deviation principle are established. Similar results were also
obtained in the non Gaussian case. See \cite{fk81} (a.e. limit for
$\lambda_1 (A)$, the earliest progress on this problem),
\cite{peche06} (CLT for general $\lambda_i(A)$, Gaussian case),
\cite{fp07} (CLT for $\lambda_1 (A)$, non Gaussian case),
\cite{cmf09} (CLT for general $\lambda_i(A)$, non Gaussian case),
\cite{maida07} (large deviation for $\lambda_1(A)$, $rank(P) = 1$,
Gaussian case), \cite{gn11} (a.e. limit for general $\lambda_i (A)$,
unitary invariant case).

Another model we considered in this paper is the spiked population
model, first proposed by \cite{johnstone01}.  Here we have
independent samples drawn from a Gaussian distribution with
covariance matrix $\Sigma$ having all but a few eigenvalues equal
one. The object under study is the ``spiked eigenvalues" of the
sample covariance matrix $S_n$. If $\Sigma=I$, then $S_n$ is a
Wishart matrix. So the spiked population model can be considered as
a finite rank perturbation of the Wishart matrix ensemble.

This model has also been extensively studied in the literature. The
asymptotic properties of the largest eigenvalues of $S_n$ were
established. The ground breaking work on this problem is
\cite{bbp05}, in which the CLT for $\lambda_i(S_n)$ was derived for
the complex Gaussian case. See also \cite{paul07} (CLT for
$\lambda_i(S_n)$, real Gaussian case), \cite{bs06} (a.e. limit for
$\lambda_i(S_n)$, non Gaussian case), \cite{karoui07} (CLT for
$\lambda_1(S_n)$, non Gaussian case), \cite{fp09} (CLT for
$\lambda_i(S_n)$, non Gaussian case).

\subsection{Main results of this paper}

Instead of considering asymptotic properties, this paper established
sharp deviation bounds for the extreme eigenvalues of matrices from
the deformed GOE and the spiked population model.

Our result about the deformed GOE is theorem \ref{GOE_Deviation}, in
which we proved
\begin{equation*}
P(|\lambda_i(A)-\lambda_{\theta_i}|\geq t)\leq C_1 e^{-C_2
nt^2/\sigma^2}
\end{equation*}
where $\lambda_i(A)$ is the $i$-th largest eigenvalue of $A=P+G$,
$\theta_i$ is the $i$-th largest eigenvalue of $P$, and
$\lambda_{\theta_i}$ is defined as
\begin{equation*}
\lambda_{\theta_i} =
\begin{cases}
\theta_i+\frac{\sigma^2}{\theta_i}  &  \quad \text{if $\theta_i > \sigma$,}\\
           2\sigma              &  \quad \text{if $0 < \theta_i \leq \sigma$.}\\
\end{cases}
\end{equation*}
Theorem \ref{GOE_Deviation} assumes that $P$ has only nonnegative
eigenvalues. A similar result for the smallest eigenvalues of $A$
holds when $P$ has negative eigenvalues.

Our results about the spiked population model are divided into two
parts. Theorem \ref{SPM_Deviation} established deviation bounds for
the largest eigenvalues, and theorem \ref{SPM_Deviation_2}
established deviation bounds for the smallest eigenvalues. We can
summarize these two theorems as the following. Let $\theta^2$ be an
eigenvalue of the population covariance matrix $\Sigma$,
$\theta^2\neq 1$. Then the corresponding ``spiked eigenvalue"
$\lambda(S_n)$ of the sample covariance matrix will satisfy
\begin{equation*}
P(|\lambda(S_n)-\lambda_{\theta,c}|\geq t)\leq C_1 e^{-C_2 nt^2}
\end{equation*}
where $\lambda_{\theta,c}$ is defined as
\begin{equation*}
\lambda_{\theta,c} =
\begin{cases}
\theta^2 + c\cdot\frac{\theta^2}{\theta^2 - 1} &\quad \text{if $\theta^2>1+\sqrt{c}$, or $c<1, \theta^2 < 1-\sqrt{c}$,}\\
(1+\sqrt{c})^2 &\quad\text{if $1 < \theta^2 \leq 1 + \sqrt{c}$,}\\
(1-\sqrt{c})^2 & \quad\text{if $c < 1, 1-\sqrt{c}\leq \theta^2 <1$.}\\
\end{cases}
\end{equation*}

Unlike the traditional approach, our method does not involve the use
moment method, Stieltjes transform, or the joint density formula for
eigenvalues. Instead, we use the min-max characterization of
eigenvalues and concentration of measure for Gaussian processes to
prove the upper tail bound for the largest eigenvalues, and use
explicit construction of eigenvectors to prove the lower tail bound.

In the existing literature, the study of the deformed GOE and the
spiked population model require completely different techniques, see
\cite{peche06} and \cite{paul07}. Our method has the advantage of
treating these two models the same way. Once the basic idea is
understood, the proof of these three theorems are almost identical.
See section 4 for an outline of proof.

\section{Notation}
$x\in\mathbb{R}^n$ is considered as a column vector, $x^*$ is the
transpose, $x_j$ is the $j$-th coordinate, $|x| = \sqrt{\sum_{i=1}^n
x_i^2}$ is the Euclidean norm. For $x,y\in\mathbb{R}^n$, let
$(x,y)=\sum_{j=1}^n x_j y_j$, $x\perp y$ means $(x,y) = 0$. $S^{n-1}
= \{x\in \mathbb{R}^n: |x|=1\}$. A metric space will be written as
$(X,d)$ where $d$ is the metric. For example, $(S^{n-1},|\cdot|)$ is
$S^{n-1}$ with the Euclidean metric. A metric will be specified
whenever we discuss $\epsilon$-net.

$\mathbb{R}^{p\times n}$ is the set of all $p\times n$ real
matrices. $\mathbb{R}^{n\times n}_{sym}$ is the set of all $n\times
n$ real symmetric matrices. For $A\in\mathbb{R}^{p\times n}$,
$\|A\|$ is the largest singular value, $A^*$ is the transpose.
$E_{i,j}$ is the matrix with $1$ on the $(i,j)$ entry and $0$
elsewhere, the ambient dimension will be clear whenever we use this
notation. The Gaussian Orthogonal Ensemble is defined as
\begin{align*}
GOE(n,\frac{\sigma^2}{n}) =& \{A\in\mathbb{R}^{n\times n}_{sym}: a_{i,j},1\leq i\leq j\leq n, \text{are independent};\\
&\quad a_{i,i}\sim \mathcal{N}(0,\frac{2\sigma^2}{n});a_{i,j}\sim\mathcal{N}(0,\frac{\sigma^2}{n}),i < j\}
\end{align*}
where $\mathcal{N}(\mu,\sigma^2)$ denotes the Gaussian distribution with mean $\mu$ and variance $\sigma^2$.

The size of a finite set $A$ will be denoted by $|A|$. For $a,b\in\mathbb{R},a\vee b = \max\{a,b\}.$

\section{Statement of Main Results}

This section contains our main results. Theorem \ref{GOE_Deviation}
is our result about the deformed GOE. Theorem \ref{SPM_Deviation}
and theorem \ref{SPM_Deviation_2} are our results about the spiked
population model.

\begin{theorem}\label{GOE_Deviation}
Let $A = P + G$, $G\in GOE(n,\frac{\sigma^2}{n})$,
$P\in\mathbb{R}^{n\times n}_{sym}$ has rank $r$ with eigenvalues
$\theta_1 \geq \cdots \geq \theta_r > 0$. Let $\lambda_1
(A)\geq\cdots\geq\lambda_n (A)$ be the eigenvalues of $A$. Define
\begin{equation}\label{GOE_Lambda_Theta}
\lambda_{\theta} =
\begin{cases}
\theta+\frac{\sigma^2}{\theta}  &  \quad \text{if $\theta > \sigma$,}\\
           2\sigma              &  \quad \text{if $0 < \theta \leq \sigma$.}\\
\end{cases}
\end{equation}

(i):  Let $C_{1,\theta}(n) =  \frac{2t(\lambda_{\theta} +
t)}{\sigma^2}\cdot n$, $C_{m,\theta}(n) = 2m C_{1,\theta}(n) (1 +
\frac{C_{1,\theta}(n)}{m-1})^{m-1}, m \geq 2$. When $r>0$, we have
\begin{equation}\label{GOE_Bound}
P(\lambda_i (A) \geq \lambda_{\theta_i} + t)\leq 2 C_{r-i+1,\theta_i} (n)\cdot e^{-\frac{(1-\delta)^2 nt^2}{4\sigma^2}}
\end{equation}
for $1\leq i \leq r$, $t\geq\frac{\sqrt{2(r-i+1)}\sigma}{\sqrt{\delta(1-\delta)n}}$, $0 < \delta \leq \frac{1}{2}$. When $r=0$, we have
\begin{equation}\label{GOE_Bound_Null_Case}
P(\lambda_1 (A)\geq 2\sigma + t)\leq e^{-\frac{nt^2}{4\sigma^2}},\quad t\geq 0
\end{equation}

(ii):  Let $r_0$ be the number of $\theta_i$ larger than $\sigma$.
If $r_0 > 0$, then for $1\leq i\leq r_0$, $t\geq 0$, we have
\begin{equation}\label{GOE_Reverse_Bound}
P(\lambda_i(A)\leq \lambda_{\theta_i}-t-C_1 \sigma r/ n)\leq
e^{-\frac{(n-r)(\theta_i - \sigma)^4}{16\sigma^2 \theta_i^2}} +
8i\cdot e^{-\frac{C_2 (n-r)(\theta_i - \sigma)^5 t^2}{\sigma^4
(\theta_i + \sigma)^3}}
\end{equation}
where $C_1,C_2$ are two positive constants. (We can pick
$C_1=2,C_2=0.25$)
\end{theorem}

We assumed $\theta_i>0$ for simplicity; theorem \ref{GOE_Deviation}
holds with trivial modification when $P$ has both positive and
negative eigenvalues.

Part (ii) of theorem \ref{GOE_Deviation} provides a lower tail bound
for $\lambda_i(A)$ only when $\theta_i > \sigma$. When
$\theta_i\leq\sigma$, we can use the semicircle law to get a lower
tail bound for $\lambda_i(A)$, this is intuitively clear: the
interval $[2\sigma-\epsilon,2\sigma]$ should contain about $\epsilon
n$ eigenvalues. See \cite{akv02} for a rigorous derivation. Our
result shows that $\lambda_i(A)$ will not exit the semicircle law
band when $\theta_i\leq\sigma$.

Theorem \ref{GOE_Deviation} essentially says, when $r$ is small, we
have $\lambda_i(A)\approx\lambda_{\theta_i},1\leq i\leq r$. As a
consequence, for fixed $r$, we have
$\lambda_i(A)\rightarrow\lambda_{\theta_i},n\rightarrow\infty$, and
the fluctuation of $\lambda_i(A)$ is of order $\frac{1}{\sqrt{n}}$.

We can also allow $r$ to grow with $n$; for example, if
$r=o(\frac{n}{\log n})$, we still have
$\lambda_i(A)\rightarrow\lambda_{\theta_i},n\rightarrow\infty$. This
can be derived by using our deviation bound and the Borel-Cantelli
lemma. The a.e. convergence of $\lambda_i(A)$ when $r$ grows like
$o(\frac{n}{\log n})$ can not be derived by existing methods in the
literature.

\begin{theorem}\label{SPM_Deviation}
$\Sigma\in\mathbb{R}^{p\times p}_{sym}$,
$\Sigma\sim\diag\{\theta_1^2,\cdots,\theta_{r+s}^2,1,\cdots,1\}$,
$\theta_1\geq\cdots\geq \theta_r > 1 >
\theta_{r+1}\geq\cdots\geq\theta_{r+s}>0$. Let
$G\in\mathbb{R}^{p\times n}$ with entries $g_{i,j}$ being i.i.d.
$\mathcal{N}(0,1)$. Consider the sample covariance matrix $S_n =
\frac{1}{n}(\Sigma^{\frac{1}{2}}G)(\Sigma^{\frac{1}{2}}G)^*$, let
$\lambda_1 (S_n)\geq\cdots\geq\lambda_p (S_n)$ be its eigenvalues.
For $c\geq 0$,define
\begin{equation}\label{SPM_Lambda_Theta_C}
\lambda_{\theta,c} =
\begin{cases}
\theta^2 + c\cdot\frac{\theta^2}{\theta^2 - 1} &\quad \text{if $\theta^2>1+\sqrt{c}$, or $c<1, \theta^2 < 1-\sqrt{c}$,}\\
(1+\sqrt{c})^2 &\quad\text{if $1 < \theta^2 \leq 1 + \sqrt{c}$,}\\
(1-\sqrt{c})^2 & \quad\text{if $c < 1, 1-\sqrt{c}\leq \theta^2 <1$.}\\
\end{cases}
\end{equation}

(i):  Let $c = \frac{p-r}{n}$, $C_{0,\theta}(n) = 1$, $C_{1,\theta}
(n) = \frac{2t(\sqrt{\lambda_{\theta,c}}+t)}{\theta^2}\cdot n$,
$C_{m,\theta}(n) = 2m C_{1,\theta}(n) (1 +
\frac{C_{1,\theta}(n)}{m-1})^{m-1}, m \geq 2$. When $r>0$, we have
\begin{equation}\label{SPM_Bound_1}
P(\sqrt{\lambda_i (S_n)}\geq \sqrt{\lambda_{\theta_i,c}} + t) \leq C_{r-i+1,\theta_i}(n)\cdot e^{-\frac{(1-\delta)^2 nt^2}{2\theta_i^2}}
\end{equation}
for $1\leq i\leq r$, $t\geq\frac{\sqrt{r-i+1}\theta_i}{\sqrt{\delta
(1-\delta)n}}$, $0<\delta \leq\frac{1}{3}$. When $r=0$, we have
\begin{equation}\label{SPM_Bound_1_Null_Case}
P(\sqrt{\lambda_1 (S_n)}\geq 1 + \sqrt{p/n} + t)\leq e^{-\frac{nt^2}{2}},\quad t\geq 0
\end{equation}

(ii): Let $r_0$ be the number of $\theta_i^2$ larger than $1 +
\sqrt{c}$, $c=\frac{p-r}{n}$. If $r_0 > 0$, then for $1\leq i \leq
r_0$, $t\geq 0$, we have
\begin{equation}\label{SPM_Reverse_Bound_1}
P(\sqrt{\lambda_i (S_n)} \leq \sqrt{\lambda_{\theta_i,c}} - t - C_1
\theta_1 r/n) \leq C_2 i\cdot e^{-C_3 nt^2}
\end{equation}
where $C_1, C_2$ are two positive constants and $C_3$ is a positive
real number depending on $\theta_1$ and $c$.
\end{theorem}

Theorem \ref{SPM_Deviation} describes the relationship between the
largest sample eigenvalues and the largest population eigenvalues.
Loosely speaking, the population eigenvalue $\theta^2$ can be
estimated by solving the equation
\begin{equation*}
\hat{\theta}^2 + c\cdot \frac{\hat{\theta}^2}{\hat{\theta}^2-1} =
\lambda(S_n)
\end{equation*}
if we observe an sample eigenvalue $\lambda(S_n) > (1+\sqrt{c})^2$;
and those population eigenvalues $\leq 1+\sqrt{p/n}$ are not
estimable from the sample covariance matrix. We also have
$\lambda_i(S_n)\rightarrow\lambda_{\theta_i,c},n\rightarrow\infty$
when $r=o(\frac{n}{\log n})$.

It's worth noting that the bounds we derived in theorem
\ref{SPM_Deviation} does not depend on population eigenvalues that
are smaller than one. This is important in applications with
heteroscedasticity. Suppose we have $n$ observations on the
variables $R_1,\cdots, R_p$ and we believe that they are driven by a
small number of principle components, i.e.
\begin{equation*}
R_i (t) = \beta_{i,1}P_1 (t) + \cdots + \beta_{i,r}P_r (t) +
\epsilon_i(t),\quad t=1,\cdots,n
\end{equation*}
Even $var(\epsilon_i)=\sigma_i^2$ are not equal, we can still
estimate the coefficients $\beta_{i,k}$ and $var(P_k)$ reliably
using the sample covariance matrix. We can estimate
$\frac{1}{p}\Sigma\sigma_i^2\approx \hat{\sigma}^2$ and simply
pretend that $\sigma_i^2=\hat{\sigma}^2$.

The proof of part (i) of theorem \ref{SPM_Deviation} was inspired by
\cite{gordon85}, in which \eqref{SPM_Bound_1_Null_Case} was proved.
\cite{gordon85} does not discuss random matrices explicitly, see
\cite{ds01} for a discussion of the results of \cite{gordon85} in
terms of random matrices.

The next theorem is about the smallest eigenvalues of the sample
covariance matrix in the spiked population model.

\begin{theorem}\label{SPM_Deviation_2}
We use the same notation as in theorem \ref{SPM_Deviation}.

(i): Assume $n>p$. Let $c = \frac{p-r}{n},c'=\frac{p-r-s}{n}$. Let
$C'_i (n) = C_{r+s-i+1,\theta_1\vee 1}(n) + C_{r,\theta_1\vee
1}(n)$. When $s>0$, we have
\begin{equation}\label{SPM_Bound_2}
P(\sqrt{\lambda_{p-i+1}(S_n)}\leq
\sqrt{\lambda_{\theta_{r+s-i+1},c'}}-\frac{\theta_1\vee 1}{2n} -t)
\leq C'_i(n) e^{-\frac{(1-\delta)^2 nt^2}{2(\theta_1^2\vee 1)}}
\end{equation}
for $1\leq i\leq s$, $t\geq\frac{\sqrt{r+s-i+1}(\theta_1\vee
1)}{\sqrt{\delta(1-\delta)n}}$, $0 < \delta\leq\frac{1}{3}$. When
$s=0,r>0$ we have
\begin{equation}\label{SPM_Bound_2_Null_Case_1}
P(\sqrt{\lambda_p (S_n)}\leq 1 -\sqrt{c'}
-\frac{\theta_1}{2n}-t)\leq 2
C_{r,\theta_1}(n)e^{-\frac{(1-\delta)^2 n t^2}{2\theta_1^2}}
\end{equation}
for
$t\geq\frac{\sqrt{r}\theta_1}{\sqrt{\delta(1-\delta)n}},0<\delta\leq\frac{1}{3}$.
When $s=0,r=0$, we have
\begin{equation}\label{SPM_Bound_2_Null_Case_2}
P(\sqrt{\lambda_p (S_n)}\leq 1 - \sqrt{p/n}-t)\leq
e^{-\frac{nt^2}{2}},\quad t\geq 0
\end{equation}

(ii): Let $s_0$ be the number of $\theta_i^2$ smaller than
$1-\sqrt{c'}$, $c'=\frac{p-r-s}{n}$. If $s_0 > 0$, then for $1\leq
i\leq s_0$, $t\geq 0$, we have
\begin{equation}\label{SPM_Reverse_Bound_2}
P(\sqrt{\lambda_{p-i+1}(S_n)}\leq
\sqrt{\lambda_{\theta_{r+s-i+1},c'}} - t - C_1 \theta_1 (r+s)/n)\leq
C_2 i \cdot e^{- C_3 nt^2}
\end{equation}
where $C_1,C_2$ are two positive constants, and $C_3$ is a positive
real number depending on $\theta_{1}$ and $c'$.
\end{theorem}

\section{Outline of Proof}
This section explains the main idea in the $r=1$ case of theorem
\ref{GOE_Deviation}.

Consider $A=\theta E_{1,1}+G,G\in
GOE(n,\frac{\sigma^2}{n}),\theta>0$, our objective is to show
$\lambda_1(A)\approx\lambda_{\theta}$.

Let's consider the upper tail bound first. If we can prove
$E[\lambda_1 (A)]\leq \lambda_{\theta}$, then the concentration of
measure for Gaussian processes will yield the desired upper tail
bound. We begin with $\lambda_1(A)=\max_{x\in S^{n-1}}x^*Ax$, so
$\lambda_1(A)$ is the maximum of the Gaussian process $\{x^*Ax: x\in
S^{n-1}\}$; and one might consider using Slepian's lemma
(proposition \ref{Slepian_1}) to prove
$E[\lambda_1(A)]\leq\lambda_{\theta}$. However, this seems to be a
tall order.

The first key idea is to stratify $S^{n-1}$ using the first
coordinate
\begin{equation*}
\lambda_1(A)=\max_{u\in [0,1]}L_u,\quad L_u=\max_{x\in
S^{n-1},x_1=u}x^*Ax
\end{equation*}
Each $L_u$ is the maximum of a Gaussian process, and we can use
Slepian's lemma to prove
\begin{equation*}
E[L_u]\leq \theta u^2 + 2\sigma\sqrt{1-u^2}=\varphi(u)
\end{equation*}
When $\theta>\sigma$, the maximum of $\varphi(u)$ is
$\varphi(\sqrt{1-\frac{\sigma^2}{\theta^2}})=\theta+\frac{\sigma^2}{\theta}$;
when $\theta\leq\sigma$, the maximum is $\varphi(0)=2\sigma$. Thus
$E[L_u]\leq\lambda_{\theta}$ and we can apply concentration of
measure for Gaussian processes to get
$P(L_u\geq\lambda_{\theta}+t)\leq e^{-\frac{nt^2}{4\sigma^2}}$.

To get an upper tail bound for $\lambda_1(A)=\max_{u\in [0,1]}L_u$,
we have to take a union bound. The second key idea is to use an
$\epsilon$-net argument to control $\lambda_1(A)$ by finitely many
$L_u$. More precisely, if $\mathscr{X}$ is an $\epsilon$-net for
$S^{n-1}$, then
\begin{equation*}
\lambda_1(A)\leq\frac{1}{1-2\epsilon}\max_{x\in\mathscr{X}}|x^*Ax|
\end{equation*}
We will use a special $\epsilon$-net
\begin{equation*}
\mathscr{X}=\cup_{u\in\mathscr{N}}\mathscr{E}(u),\quad
\mathscr{E}(u) = \{x\in S^{n-1}: x_1=u\}
\end{equation*}
where $\mathscr{N}$ is a finite subset of $[0,1]$ whose size depend
on $\epsilon$. Then we can use $P(L_u\geq\lambda_{\theta}+t)\leq
e^{-\frac{nt^2}{4\sigma^2}}$ and
\begin{equation*}
\lambda_1(A) \leq\frac{1}{1-2\epsilon}\max_{u\in\mathscr{N}}L_u
\end{equation*}
to build an upper tail bound for $\lambda_1(A)$. The final step is
optimizing over $\epsilon$ to get the best bound ($\epsilon$ should
be of order $\frac{1}{n}$).

Now, let's consider the lower tail bound for $\lambda_1(A)$. We want
to construct an $x\in S^{n-1}$ with $x^* Ax\approx\lambda_{\theta}$.
(To be precise, we want a bound for $P(x^*Ax\leq
\lambda_{\theta}-t)$) Consider
\begin{equation*}
(\theta E_{1,1}+G)x = \lambda_{\theta}x, x\in S^{n-1}
\end{equation*}
Let $\tilde{G}$ be the lower right $(n-1)\times(n-1)$ submatrix of
$G$, $v = (g_{2,1},\cdots,g_{n,1})^t$, $\tilde{x} =
(x_2,\cdots,x_n)^t$, then the above equation becomes
\begin{equation*}
\theta x_1 + v^* \tilde{x} = \lambda_{\theta}x_1,\quad x_1
v+\tilde{G}\tilde{x}=\lambda_{\theta}\tilde{x}
\end{equation*}
Thus $\tilde{x} = -x_1(\tilde{G}-\lambda_{\theta}I)^{-1}v$. Of
course, such an $x$ might not be in existence at all, since
$\lambda_{\theta}$ might not be an eigenvalue of $A$. However, this
heuristic ``Schur complement" argument suggests a way to construct
approximate eigenvectors for $A$. When $\theta
> \sigma$, we know the correct value for $x_1$ is
$\sqrt{1-\frac{\sigma^2}{\theta^2}}$ (since $\varphi(u)$ attains its
maximum $\lambda_{\theta}$ at this point). So the correct way to
construct approximate eigenvector is
\begin{equation*}
x_1 = \sqrt{1-\frac{\sigma^2}{\theta^2}},\quad \tilde{x} = - c
(\tilde{G}-\lambda_{\theta}I)^{-1}v,\quad c>0, |x|=1
\end{equation*}
With this $x$, we have $x^* Ax\approx\lambda_{\theta}$. In fact, the
formula for $x^* Ax$ involves $L_1 = v^*Rv$ and $L_2 = v^*R^2 v$,
where $R= (\tilde{G}-\lambda_{\theta}I)^{-1}$; and we can use
Wigner's semicircle law to show $L_1 \approx \frac{1}{n}tr R\approx
-\frac{1}{\theta}$, $L_2\approx \frac{1}{n}tr R^2\approx
\frac{1}{\theta^2-\sigma^2}$. After some straightforward
calculation, this gives us $x^*Ax \approx\lambda_{\theta}$.

\section{Deformed GOE: Proof of Theorem \ref{GOE_Deviation}}

By the orthogonal invariance property of the $GOE$, we can assume $P
= \theta_1 E_{1,1} + \cdots + \theta_r E_{r,r}$. The proof is
divided into two subsections, corresponding to the two parts of
theorem \ref{GOE_Deviation}.

\subsection{Upper Tail Bound for the Largest Eigenvalues}

We prove part (i) of theorem \ref{GOE_Deviation} in this section.

The $r > 0$ case. By the minimax characterization of eigenvalues,
for $1\leq i\leq r$, we have
\begin{align}\label{GOE_Control_By_Norm}
\lambda_i (A) &= \min_{x^1,\cdots,x^{i-1}\in\mathbb{R}^n} \max_{x\in
S^{n-1}\cap\{x^1,\cdots,x^{i-1}\}^{\perp}}x^*Ax\notag\\
&\leq\max_{x\in S^{n-1},x_1=\cdots =x_{i-1}=0} x^*Ax = \|A_i\|
\end{align}
$A_i$ is the lower right $(n-i+1)\times (n-i+1)$ part of $A$. We can
consider $A_i$ as a linear operator from $V_i$ to itself, where $V_i
= \{x\in\mathbb{R}^n: x_1=\cdots=x_{i-1}=0\}$, so that its operator
norm can be controled by the maximum of $|x^*Ax|$ over an
$\epsilon$-net. More precisely, using lemma
\ref{Epsilon_Net_Lemma_Eigenvalue}, we have
\begin{equation}\label{GOE_Epsilon_Discretization}
\|A_i\|\leq\frac{1}{1-2\epsilon}\max\{x^*Ax,-x^*Ax: x\in\mathscr{X}_i\},\quad 0 < \epsilon < \frac{1}{2}
\end{equation}
When $i=r$, $\mathscr{X}_r$ is an $\epsilon$-net for $(S^{n-1}\cap
V_r \cap \{x\in\mathbb{R}^n: x_r\geq 0 \},|\cdot|)$. When $1\leq i
\leq r-1$, $\mathscr{X}_i$ is an $\epsilon$-net for $(S^{n-1}\cap
V_i, |\cdot|)$. (In the $r=1$ case, we do not have the $1\leq i \leq
r-1$ part.)

Our proof uses $\epsilon$-nets with a special structure, whose
construction and bound for its size are rather delicate, so we defer
the details to the appendix. See lemma
\ref{Epsilon_Net_Construction_1} and lemma
\ref{Epsilon_Net_Construction_2}.

Let's consider $\|A_r\|$ first. In this case
\begin{equation}\label{GOE_Epsilon_Net_Structure_1}
\mathscr{X}_r = \cup_{u\in\mathscr{N}_r} \mathscr{E}_r(u)
\end{equation}
where $\mathscr{E}_r (u) = \{x\in S^{n-1}: x_1 = \cdots = x_{r-1} =
0, x_r = u\}, u\in [0,1]$, $\mathscr{N}_r$ is a finite subset of
$[0,1]$ with $|\mathscr{N}_r|\leq\frac{2}{\epsilon}$. When $1\leq
i\leq r-1$
\begin{equation}\label{GOE_Epsilon_Net_Structure_2}
\mathscr{X}_i = \cup_{u\in\mathscr{N}_i}\mathscr{E}_i (u)
\end{equation}
where $\mathscr{E}_i (u) = \{x\in S^{n-1}: x_1 = \cdots = x_{i-1} =
0, (x_i,\cdots,x_r) = u\}, u\in B^{r-i+1}_2$, $\mathscr{N}_i$ is a
finite subset of $B^{r-i+1}_2$ with $|\mathscr{N}_i|\leq
\frac{4(r-i+1)^2}{\epsilon}(1 +
\frac{2(r-i+1)}{(r-i)\epsilon})^{r-i}$.

With this structure for $\mathscr{X}_i$, maximizing $|x^*Ax|$ over
$x\in\mathscr{X}_i$ becomes stratified: we can maximize $|x^*Ax|$
for $x\in\mathscr{E}_i (u)$ to get $L_{i,u} =
\max_{x\in\mathscr{E}_i (u)}x^*Ax$, $\tilde{L}_{i,u} =
\max_{x\in\mathscr{E}_i (u)} -x^*Ax$ for each $u\in\mathscr{N}_i$,
then select the largest among
$L_{i,u},\tilde{L}_{i,u},u\in\mathscr{N}_i$. i.e.
\begin{equation}\label{GOE_Stratified_Maximization}
\max\{x^*Ax, -x^*Ax: x\in\mathscr{X}_i\} = \max\{L_{i,u}, \tilde{L}_{i,u}: u\in\mathscr{N}_i\}
\end{equation}
\eqref{GOE_Control_By_Norm}, \eqref{GOE_Epsilon_Discretization} and \eqref{GOE_Stratified_Maximization} imply
\begin{equation}\label{GOE_Lambda_By_Epsilon}
\lambda_{i}(A) \leq \frac{1}{1-2\epsilon}\max\{L_{i,u},\tilde{L}_{i,u}: u\in\mathscr{N}_i\},\quad 0 < \epsilon < \frac{1}{2}
\end{equation}

\eqref{GOE_Lambda_By_Epsilon} is the starting point for building an
upper tail bound for $\lambda_i (A)$. The next step is to eatablish
a tail bound for each $L_{i,u},\tilde{L}_{i,u}$, then take a union
bound over $u\in\mathscr{N}_i$. We keep $\epsilon$ as a free
parameter along the way and optimize over $\epsilon$ at the end. To
get an upper tail bound for $L_{i,u}$ (similarly for
$\tilde{L}_{i,u}$), we will prove $E[L_{i,u}]\leq
\lambda_{\theta_i}$ using Slepian's lemma as stated below; then use
concentration of measure inequality for Gaussian processes.

\begin{proposition}\emph{(Slepian's Lemma)}\label{Slepian_1}
Let $(X_t)_{t\in T}$ and $(Y_t)_{t\in T}$ be two centered Gaussian
processes defined on the same finite index set $T$. Assume $E|X_s -
X_t|^2\leq E|Y_s - Y_t|^2$ for all $s,t\in T$. Then $E[\max_{t\in
T}X_t] \leq E[\max_{t\in T}Y_t]$.
\end{proposition}

\begin{remark}
Although Slepian's lemma is stated for Gaussian processes defined on
a finite index set, we will use it on Gaussian processes defined on
infinite index sets. This is justified by an approximation procedure
and we omit this routine matter. This remark applies also to the
application of proposition \ref{Minimax_Gaussian_Concentration} and
proposition \ref{Gordon_1}.

A proof of proposition \ref{Slepian_1} and its generalization (proposition \ref{Gordon_1} below) can be found in \cite{gordon85}.
\end{remark}

Let $X_x = x^*Gx,x\in S^{n-1}$, write $u = (u_i,\cdots,u_r)\in
B^{r-i+1}_2$, then
\begin{align}
&L_{i,u} = \sum_{j = i}^{r}\theta_j u_j^2 + \max_{x\in\mathscr{E}_i (u)} X_x \label{GOE_Usual_Lu}\\
&\tilde{L}_{i,u} = -\sum_{j=i}^r \theta_j u_j^2 + \max_{x\in\mathscr{E}_i (u)} -X_x \label{GOE_Tilde_Lu}
\end{align}
Let $Y_x = \frac{2\sigma}{\sqrt{n}}\sum_{j = r+1}^{n} x_j \omega_j,
x\in S^{n-1}$, where $\omega_{r+1},\cdots,\omega_n$ are i.i.d.
$\mathcal{N}(0,1)$ random variables. Then for $x,y\in\mathscr{E}_i
(u)$, we have
{\allowdisplaybreaks\begin{align}\label{GOE_Slepian_Comparison}
E|X_x - X_y|^2 &= E|\sum_{k,j=1}^n (x_k x_j - y_k y_j)g_{k,j}|^2\\
&=\sum_{k=1}^n (x_k^2 - y_k^2)^2\frac{2\sigma^2}{n} + \sum_{1\leq k < j\leq n}4(x_k x_j - y_k y_j)^2\frac{\sigma^2}{n}\notag\\
&= \frac{2\sigma^2}{n}((\sum_{k=1}^n x_k^2)^2 + (\sum_{k=1}^n y_k^2)^2 - 2(\sum_{k=1}^n x_k y_k)^2)\notag\\
&= \frac{2\sigma^2}{n}(2 - 2(x,y)^2)\notag\\
&= \frac{4\sigma^2}{n}|x-y|^2 - \frac{4\sigma^2}{n}(1 - (x,y))^2\notag\\
&\leq  \frac{4\sigma^2}{n}|x-y|^2 = E|Y_x - Y_y|^2\notag
\end{align}}
Using Slepian's lemma, and noticing that the maximum of $\{Y_x:
x\in\mathscr{E}_i (u)\}$ is reached when $(x_{r+1},\cdots,x_n)$ is a
multiple of $(\omega_{r+1},\cdots,\omega_n)$, we have
\begin{align*}
&E[\max_{x\in\mathscr{E}_i (u)}X_x], \quad E[\max_{x\in\mathscr{E}_i (u)}-X_x] \\
&\leq E[\max_{x\in\mathscr{E}_i (u)}Y_x]= E[\frac{2\sigma\sqrt{1-|u|^2}}{\sqrt{n}}\sqrt{\sum_{j=r+1}^n \omega_j^2}]\\
&\leq \frac{2\sigma\sqrt{1-|u|^2}}{\sqrt{n}}\sqrt{E[\sum_{j=r+1}^n \omega_j^2]}\quad\text{(by Cauchy-Schwarz)}\\
&= 2\sigma\sqrt{1-|u|^2}\sqrt{1-\frac{r}{n}}\leq 2\sigma\sqrt{1-|u|^2}
\end{align*}
This together with \eqref{GOE_Usual_Lu}, \eqref{GOE_Tilde_Lu} imply
\begin{align*}
&E[L_{i,u}]\leq \sum_{j=i}^r \theta_j u_j^2 + 2\sigma\sqrt{1-|u|^2} = \varphi(u)\\
&E[\tilde{L}_{i,u}]\leq -\sum_{j=i}^r\theta_j u_j^2 + 2\sigma\sqrt{1-|u|^2} \leq \varphi(u)
\end{align*}

When $\theta_i > \sigma$, the maximum of $\varphi(u)$ is $\theta_i +
\frac{\sigma^2}{\theta_i}$ and is attained when $u = (\sqrt{1 -
\frac{\sigma^2}{\theta_i^2}},0,\cdots,0)$; when $\theta_i
\leq\sigma$, the maximum of $\varphi(u)$ is $2\sigma$. Therefore,
the previous two inequalities imply
\begin{equation}\label{GOE_Lu_Expectation}
E[L_{i,u}], \quad E[\tilde{L}_{i,u}] \leq \lambda_{\theta_i}
\end{equation}
Now we want to apply concentration inequalities to
$L_{i,u},\tilde{L}_{i,u}$, which are sumprema of Gaussian processes.
We need the following proposition, which is proved via the Gaussian
concentration inequality (proposition \ref{Gaussian_Concentration}
in the appendix).

\begin{proposition}\label{Minimax_Gaussian_Concentration}
Let $(X_{i,j})_{1\leq i\leq n,1\leq j\leq m}$ be a centered Gaussian process. Then for $t\geq 0$
\begin{align*}
P(\min_{i}\max_{j}X_{i,j} \geq E[\min_{i}\max_{j}X_{i,j}] + t)\leq e^{\frac{-t^2}{2\max_{i,j}EX_{i,j}^2}}\\
P(\min_{i}\max_{j}X_{i,j} \leq E[\min_{i}\max_{j}X_{i,j}] - t)\leq e^{\frac{-t^2}{2\max_{i,j}EX_{i,j}^2}}
\end{align*}
\end{proposition}
\begin{proof}
We can find i.i.d. $\mathcal{N}(0,1)$ random variables $Y_1,\cdots,
Y_{nm}$ and $A\in\mathbb{R}^{nm\times nm}$ so that $X_{i,j} =
\sum_{k=1}^{nm}a_{(i-1)m+j, k}Y_k, 1\leq i\leq n, 1\leq j\leq m$.
Define $g_{i,j}(y) = \sum_{k=1}^{nm}a_{(i-1)m+j, k}y_k, g(y) =
\min_{1\leq i\leq n}\max_{1\leq j\leq m}g_{i,j}(y),
y\in\mathbb{R}^{nm}$. Then $\min_i \max_j X_{i,j} = g(Y)$.

Let $y,\tilde{y}\in\mathbb{R}^{nm}$, $g(y) =
g_{i_1,j_1}(y),g(\tilde{y})=g_{i_2,j_2}(\tilde{y})$. Assume
$g(y)\geq g(\tilde{y})$, $\max_{1\leq j\leq m}g_{i_2,j}(y) =
g_{i_2,j_3}(y)$, then
\begin{align*}
|g(y) - g(\tilde{y})| &= g_{i_1,j_1}(y) - g_{i_2,j_2}(\tilde{y}) = \min_{1\leq i\leq n}\max_{1\leq j\leq m}g_{i,j}(y) - g_{i_2,j_2}(\tilde{y})\\
&\leq \max_{1\leq j\leq m}g_{i_2,j}(y) - g_{i_2,j_2}(\tilde{y}) = g_{i_2,j_3}(y)-\max_{1\leq j\leq m}g_{i_2,j}(\tilde{y})\\
&\leq g_{i_2,j_3}(y) - g_{i_2,j_3}(\tilde{y}) \leq \max_i\max_j|g_{i,j}(y) - g_{i,j}(\tilde{y})|
\end{align*}
Hence $g$ has Lipschitz constant bounded by the norm of the operator
$A:(\mathbb{R}^{nm},l_2) \rightarrow (\mathbb{R}^{nm},l_{\infty})$,
which equals
\begin{equation*}\max_{i,j} \sqrt{\sum_{k=1}^{nm}a_{(i-1)m+j,k}^2} = \sqrt{\max_{i,j} EX_{i,j}^2}
\end{equation*}
Then we can apply proposition \ref{Gaussian_Concentration} to conclude.
\end{proof}
\begin{remark}
Using a similar arguemnt, proposition
\ref{Minimax_Gaussian_Concentration} generalizes naturally to:
$\max\min$, $\max\min\max$, $\min\max\min$, etc. We will only use
the ``max'' version (i.e. $n=1$ case) in the argument that follows.
However, we will be using the full ``minmax'' version in the proof
of theorem \ref{SPM_Deviation}.
\end{remark}

Since $E[X_x^2] = \sum_{k=1}^n x_k^4 E[g_{k,k}^2] + \sum_{k <
j}4x_k^2 x_j^2 E[g_{k,j}^2] = \frac{2\sigma^2}{n},x\in S^{n-1}$, we
can use \eqref{GOE_Lu_Expectation} and apply proposition
\ref{Minimax_Gaussian_Concentration} to get
\begin{align*}
&P(L_{i,u}\geq \lambda_{\theta_i} + t)\leq e^{-\frac{nt^2}{4\sigma^2}},\quad t\geq 0\\
&P(\tilde{L}_{i,u}\geq \lambda_{\theta_i} + t)\leq e^{-\frac{nt^2}{4\sigma^2}},\quad t\geq 0
\end{align*}
Let $\epsilon = \frac{(1-a)t}{2(\lambda_{\theta_i} + t)}, 0 < a <
1$, then $(1-2\epsilon)(\lambda_{\theta_i} + t) = \lambda_{\theta_i}
+ at$. Using \eqref{GOE_Lambda_By_Epsilon} and the above two
inequalities, we have
\begin{align*}
P(\lambda_i (A)\geq \lambda_{\theta_i} + t) &\leq P(\max\{L_{i,u},\tilde{L}_{i,u}: u\in\mathscr{N}_i\}\geq (1-2\epsilon)(\lambda_{\theta_i} + t))\\
&\leq \sum_{u\in\mathscr{N}_i}P(L_{i,u}\geq \lambda_{\theta_i}+ at) +\sum_{u\in\mathscr{N}_i}P(\tilde{L}_{i,u}\geq\lambda_{\theta_i} + at)\\
&\leq 2|\mathscr{N}_i|e^{-\frac{n a^2 t^2}{4\sigma^2}}
\end{align*}
When $t\geq\frac{\sqrt{2(r-i+1)}\sigma}{\sqrt{\delta (1-\delta)n}}$,
choose $a = \frac{1}{2}(1 + \sqrt{1 - \frac{8(r-i+1)\sigma^2}{n
t^2}})$, then $a \geq 1- \delta\geq\frac{1}{2}$. This choice of $a$
guarantees $0 < \epsilon \leq\frac{1}{3}$ (This is needed when we
apply \eqref{GOE_Lambda_By_Epsilon} and lemma
\ref{Epsilon_Net_Construction_1}). When $i=r$, we use
$|\mathscr{N}_r|\leq\frac{2}{\epsilon}$; when $1\leq i\leq r-1$, we
use $|\mathscr{N}_i|\leq\frac{4(r-i+1)^2}{\epsilon}(1 +
\frac{2(r-i+1)}{(r-i)\epsilon})^{r-i}$. After some simplification,
we get
\begin{equation*}
P(\lambda_i (A)\geq \lambda_{\theta_i} + t)\leq 2 C_{r-i+1,\theta_i}(n)\cdot e^{-\frac{na^2 t^2}{4\sigma^2}}\leq 2 C_{r-i+1,\theta_i} (n)\cdot e^{-\frac{(1-\delta)^2 nt^2}{4\sigma^2}}
\end{equation*}
This finishes the proof of the $r > 0$ case.

When $r=0$, $\lambda_1 (A) = \max_{x\in S^{n-1}}X_x$ with $X_x =
x^*Gx$. Define $Y_x = \frac{2\sigma}{\sqrt{n}}\sum_{j=1}^n
x_j\omega_j, x\in S^{n-1}$, $\omega_1,\cdots,\omega_n$ are i.i.d.
$\mathcal{N}(0,1)$ random variables. Similar to
\eqref{GOE_Slepian_Comparison}, we have $E|X_x - X_y|^2\leq E|Y_x -
Y_y|^2, \forall x, y\in S^{n-1}$; thus $E[\lambda_1 (A)]\leq
E[\max_{x\in S^{n-1}}Y_x]\leq 2\sigma$ by Slepian's lemma. Applying
proposition \ref{Minimax_Gaussian_Concentration}, we get
\begin{equation*}
P(\lambda_1 (A)\geq 2\sigma + t)\leq e^{-\frac{nt^2}{4\sigma^2}},\quad t\geq 0
\end{equation*}

\subsection{Approximate Eigenvectors}

We prove part (ii) of theorem \ref{GOE_Deviation} in this section.

The idea of the proof is to construct, for each $i,1\leq i\leq r_0$,
an approximate eigenvector $x$, i.e. $x\in S^{n-1}$, with
$x^*Ax\approx\lambda_{\theta_i}$.

Let $m=n-r$, $\sqrt{\frac{m}{n}}\tilde{G}$ be the lower right
$m\times m$ submatrix of $G$, then $\tilde{G}\in
GOE(m,\frac{\sigma^2}{m})$. Let $2\sigma < \lambda_0 <
\lambda_{\theta_i}$, $B =
\{\lambda_{\max}(\tilde{G})\leq\lambda_0\}$, then
\eqref{GOE_Bound_Null_Case} in (i) implies
\begin{equation}\label{GOE_Truncation_Error_Prob}
P(B^c) \leq e^{-\frac{m(\lambda_0 - 2\sigma)^2}{4\sigma^2}}
\end{equation}
We will only construct approximate eigenvectors on the event $B$;
the indicator $1_B$ might not be mentioned at every instance.

Let $x\in S^{n-1}$ be such that $x_1 = \cdots = x_{i-1} = x_{i+1} =
\cdots = x_r = 0$, $x_i = \sqrt{1-\frac{\sigma^2}{\theta_i^2}}$, and
\begin{align}\label{GOE_Schur_Complement_Eigenvector}
&(x_{r+1},\cdots,x_n)^t = -\frac{\sigma}{\theta_i}\frac{Rv}{\sqrt{L_2}}\\
&(g_{i,r+1},\cdots,g_{i,n})^t = \sqrt{\frac{m}{n}}\sigma v\notag
\end{align}
The random vector $v$ defined above has i.i.d.
$\mathcal{N}(0,\frac{1}{m})$ coordinates. $R = (\tilde{G} -
\lambda_{\theta_i}I)^{-1}$, $L_1 = v^* R v$, $L_2 = v^* R^2 v$.  A
straight forward calculation shows
\begin{align}\label{GOE_Lambda_Minus_xAx}
\lambda_{\theta_i} - x^*Ax &= (1 - \sqrt{1-\frac{r}{n}})\frac{2\sigma^2}{\theta_i} - g_{i,i}(1 - \frac{\sigma^2}{\theta_i^2})\\
&\quad +
\sqrt{\frac{m}{n}}\frac{\sigma^2}{\theta_i^2}(\frac{-L_1}{L_2} -
\frac{\theta_i^2 -\sigma^2}{\theta_i})\notag\\
&\quad+\sqrt{\frac{m}{n}}\frac{2\sigma^2}{\theta_i}\sqrt{1-\frac{\sigma^2}{\theta_i^2}}(\frac{L_1}{\sqrt{L_2}}+\frac{\sqrt{\theta_i^2-\sigma^2}}{\theta_i})\notag
\end{align}
The next step is to show $L_1\approx -\frac{1}{\theta_i}, L_2\approx
\frac{1}{\theta_i^2 - \sigma^2}$.  This makes the grouping of terms
in \eqref{GOE_Lambda_Minus_xAx} clear:  the four terms are all small
and we can take a union bound to get a deviation inequality for
$\lambda_{\theta_i} - x^*Ax$.

There are two sources of randomness in $L_j$: $v$ and $\tilde{G}$;
and they are independent.  We break the task of building deviation
inequalities for $L_j$ into three steps.  The first step is to show
that, conditioning on $\tilde{G}$, $L_j$ concentrates around $E[L_j
| \tilde{G}]$, see lemma \ref{GOE_L_Concentration}. The second step
is to show that $E[L_j | \tilde{G}]$ concentrates around $E[L_j]$,
see lemma \ref{GOE_Trace_Concentration}.  The third step is to show
$E[L_1]\approx -\frac{1}{\theta_i},E[L_2]\approx\frac{1}{\theta_i^2
- \sigma^2}$, see lemma \ref{GOE_Trace_Expectation_Distance}.

\begin{lemma}\label{GOE_L_Concentration}
$t\geq 0$. On the event $B$, we have
\begin{align*}
&P(L_1 - \frac{1}{m}tr R \leq -t\ |\ \tilde{G})\leq e^{-\frac{1}{4}m(\sqrt{1 + 2(\lambda_{\theta_i} - \lambda_0)t} - 1)^2}\\
&P(L_1 - \frac{1}{m}tr R \geq t\ |\ \tilde{G})\leq e^{-\frac{1}{4}m(\lambda_{\theta_i}-\lambda_0)^2 t^2}\\
&P(L_2 - \frac{1}{m}tr R^2 \leq -t\ |\ \tilde{G})\leq e^{-\frac{1}{4}m (\lambda_{\theta_i}-\lambda_0)^4 t^2}\\
&P(L_2 - \frac{1}{m}tr R^2 \geq t\ |\ \tilde{G})\leq
e^{-\frac{1}{4}m(\sqrt{1 + 2(\lambda_{\theta_i} - \lambda_0)^2 t} -
1)^2}
\end{align*}
\end{lemma}

Conditioning on $\tilde{G}$, $R$ is a constant matrix. Since the
distribution of $v$ is orthogonal invariant, we can diagonalize the
quadratic forms $L_1 = v^* Rv, L_2 = v^* R^2 v$ and apply
proposition \ref{Chi_Square_Tail_Bound} to get lemma
\ref{GOE_L_Concentration}. The bounds in the first and fourth
inequalities are complicated. When we apply lemma
\ref{GOE_L_Concentration}, we will use $e^{-\frac{1}{4}m(1-\delta)^2
(\lambda_{\theta_i}-\lambda_0)^2 t^2}$ as a bound in the first
inequality; this is valid when
$(\lambda_{\theta_i}-\lambda_0)t\leq\frac{2\delta}{(1-\delta)^2}$.
Similarly for the fourth inequality.

\begin{lemma}\label{GOE_Trace_Concentration}
$t\geq 0$. On the event $B$, we have
\begin{align*}
&P(\frac{1}{m}tr R -E[\frac{1}{m}(tr R)1_B]\geq t \ (or \leq -t)\
)\leq
e^{-\frac{1}{2}m(\lambda_{\theta_i}-\lambda_0)^2 t^2}\\
&P(\frac{1}{m}tr R^2 -E[\frac{1}{m}(tr R^2)1_B]\geq t \ (or \leq
-t)\  )\leq e^{-\frac{1}{2}m(\lambda_{\theta_i}-\lambda_0)^4 t^2}
\end{align*}
\end{lemma}

\begin{proof}
$f = \frac{1}{m}(tr R)1_B$ is a function of the random variables
$g_{l,j}$; and these $g_{l,j}$ are independent. If we can divide
$\{g_{l,j}\}$ into several groups such that each group has limited
influence on $f$, then McDiarmid's Inequality (proposition
\ref{McDiarmid_Inequality}) will yield a concentration inequality
for $f$.

Let $\tilde{G}_j$ be the submatrix of $\tilde{G}$ obtained by
deleting the $j$-th row and $j$-th column. Then proposition
\ref{Cauchy_Interlacing_Law} implies
\begin{equation}\label{GOE_G_Interlacing}
\lambda_1 (\tilde{G})\geq\lambda_1
(\tilde{G}_j)\geq\lambda_2(\tilde{G}) \geq \lambda_2
(\tilde{G}_j)\geq\cdots\geq\lambda_m(\tilde{G})
\end{equation}
Let $\varphi(x) = \frac{1}{\lambda_{\theta_i}-x}, x\in
(-\infty,\lambda_0]$. We have
\begin{align*}
&-\frac{1}{m}(tr R)1_B + \frac{1}{m}(tr (\tilde{G}_j - \lambda_{
\theta_i})^{-1})1_B\\
&= \frac{1}{m}[\sum_{l=1}^m \varphi(\lambda_l(\tilde{G})) -
\sum_{l=1}^{m-1}\varphi(\lambda_l(\tilde{G}_j))]1_B\\
&= \frac{1}{m}[\sum_{l=1}^{m-1}(\varphi(\lambda_l(\tilde{G})) -
\varphi(\lambda_l(\tilde{G}_j)) ) + \varphi(\lambda_m
(\tilde{G}))]1_B
\end{align*}
The monotonicity of $\varphi(x)$ and \eqref{GOE_G_Interlacing} imply
\begin{align*}
&|\frac{1}{m}[\sum_{l=1}^{m-1}(\varphi(\lambda_l(\tilde{G})) -
\varphi(\lambda_l(\tilde{G}_j)) ) + \varphi(\lambda_m
(\tilde{G}))]1_B|\\
&\leq\frac{1}{m}\varphi(\lambda_1(\tilde{G}))1_B\leq\frac{1}{m(\lambda_{\theta_i}-\lambda_0)}
\end{align*}
So $f$ is within $\frac{1}{m(\lambda_{\theta_i}-\lambda_0)}$
distance to $\frac{1}{m}(tr (\tilde{G}_j - \lambda_{
\theta_i})^{-1})1_B$. Therefore, changing the $j$-th row and $j$-th
column of $\tilde{G}$ will leave $f$ to vary in an interval of
length $\frac{2}{m(\lambda_{\theta_i}-\lambda_0)}$.

Divide $\{g_{l,j}\}$ into $m$ groups: $X_s=\{g_{l,j}|
\max(l,j)=s\},s = r+1,\cdots,n$. Then each $X_s$ influences $f$ by
at most $\frac{2}{m(\lambda_{\theta_i}-\lambda_0)}$.  Now we can
apply proposition \ref{McDiarmid_Inequality} to get the first
inequality. The proof for the second inequality is similar.
\end{proof}

Let
$g(z)=\int_{-2\sigma}^{2\sigma}\frac{1}{x-z}\frac{\sqrt{4\sigma^2 -
x^2}}{2\pi\sigma^2}dx$ be the Stieltjes transform of the semicircle
law.  Then
\begin{equation}\label{GOE_Stieltjes_At_Lambda_Theta}
g(\lambda_{\theta_i}) = -\frac{1}{\theta_i},\quad
g'(\lambda_{\theta_i}) = \frac{1}{\theta_i^2 - \sigma^2}
\end{equation}

\begin{lemma}\label{GOE_Trace_Expectation_Distance}
There exists a constant $C$ such that
\begin{align*}
&|E[\frac{1}{m}(tr R)1_B] - g(\lambda_{\theta_i})|
\leq\frac{C}{m(\lambda_{\theta_i}-\lambda_0)}\\
&|E[\frac{1}{m}(tr R^2)1_B] - g'(\lambda_{\theta_i})|
\leq\frac{C}{m(\lambda_{\theta_i} - \lambda_0)^2}
\end{align*}
\end{lemma}

\begin{proof}
This lemma is proved using lemma \ref{GOE_Convergence_Rate}, we will
adopt the notation of lemma \ref{GOE_Convergence_Rate} in the
following. Let $\varphi(x) =
\frac{1}{\lambda_{\theta_i}-x}1_{x\leq\lambda_0}$, the first
inequality can be reformulated as
\begin{equation}\label{Used_In_Proving_Lemmas_1}
|\int\varphi(x)dEF_m(x) - \int\varphi(x)dF(x)|\leq
\frac{C}{m}\|\varphi\|_{\max}
\end{equation}

Lemma \ref{GOE_Convergence_Rate} says
\begin{equation*}
|\int\phi(x)dEF_m(x) - \int\phi(x)dF(x)|\leq
\frac{C}{m}\|\phi\|_{\max}
\end{equation*}
if $\phi(x)=c\cdot 1_{(-\infty,a]}(x)$. To prove
\eqref{Used_In_Proving_Lemmas_1}, we approximate $\varphi(x)$ by
$\varphi_{\Delta}(x) = \sum c_i\cdot 1_{(-\infty,a_i]}(x)$, where
$\Delta$ is the division $a_1 < \cdots < a_k = \lambda_0$, and the
coefficients are $c_k = \varphi(a_k)$,$c_i =
\varphi(a_i)-\varphi(a_{i+1}), i\leq k-1$. Since $\varphi(x)$ is
nonnegative and monotone, we have $\sum |c_i|\leq
2\|\varphi\|_{\max}$, thus
\begin{equation*}
|\int\varphi_{\Delta}(x)dEF_m(x) - \int\varphi_{\Delta}(x)dF(x)|\leq
\sum \frac{C}{m}|c_i|\leq\frac{2C}{m}\|\varphi\|_{\max}
\end{equation*}
Let $\|\Delta\|=\max(a_i - a_{i-1})\rightarrow 0$, we get
\eqref{Used_In_Proving_Lemmas_1}. The second inequality is proved in
a similar fashion.
\end{proof}

Combine the previous three lemmas, we can establish deviation
inequalities for $L_1$ and $L_2$ as follows.  When $0\leq
(\lambda_{\theta_i}-\lambda_0)t\leq\frac{\sqrt{2}\delta(1+\sqrt{2}-\delta)}{(1-\delta)^2}$
\begin{align}\label{GOE_L1_Deviation_Inequality}
&P(L_1-g(\lambda_{\theta_i})\geq t +
\frac{C}{m(\lambda_{\theta_i}-\lambda_0)}) \leq
2e^{-\frac{1}{2}m(\lambda_{\theta_i}-\lambda_0)^2(\frac{1-\delta}{1+\sqrt{2}-\delta})^2
t^2}\\
&P(L_1-g(\lambda_{\theta_i})\leq -t -
\frac{C}{m(\lambda_{\theta_i}-\lambda_0)}) \leq
2e^{-\frac{1}{2}m(\lambda_{\theta_i}-\lambda_0)^2(\frac{1-\delta}{1+\sqrt{2}-\delta})^2
t^2}\notag
\end{align}
When $0\leq
(\lambda_{\theta_i}-\lambda_0)^2t\leq\frac{\sqrt{2}\delta(1+\sqrt{2}-\delta)}{(1-\delta)^2}$
\begin{align}\label{GOE_L2_Deviation_Inequality}
&P(L_2-g'(\lambda_{\theta_i})\geq t +
\frac{C}{m(\lambda_{\theta_i}-\lambda_0)^2}) \leq
2e^{-\frac{1}{2}m(\lambda_{\theta_i}-\lambda_0)^4(\frac{1-\delta}{1+\sqrt{2}-\delta})^2
t^2}\\
&P(L_2-g'(\lambda_{\theta_i})\leq -t -
\frac{C}{m(\lambda_{\theta_i}-\lambda_0)^2}) \leq
2e^{-\frac{1}{2}m(\lambda_{\theta_i}-\lambda_0)^4(\frac{1-\delta}{1+\sqrt{2}-\delta})^2
t^2}\notag
\end{align}

The first two terms in \eqref{GOE_Lambda_Minus_xAx} are bounded
above by $\frac{2r\sigma}{n}+|g_{i,i}|$.  Using
\eqref{GOE_Stieltjes_At_Lambda_Theta},
\eqref{GOE_L1_Deviation_Inequality} and
\eqref{GOE_L2_Deviation_Inequality}, we can build deviation bounds
for the last two terms in \eqref{GOE_Lambda_Minus_xAx} (the routine
details are omitted). We choose
$\lambda_0=\frac{1}{2}(2\sigma+\lambda_{\theta_i})$. The end result
is
\begin{equation}\label{GOE_Lambda_Minus_xAx_Bound}
P(\{\lambda_{\theta_i}-x^* Ax\geq t + C_1 \sigma r/n\}\cap B)\leq 8
e^{-\frac{C_2 m (\theta_i - \sigma)^5 t^2}{\sigma^4 (\theta_i +
\sigma)^3}}
\end{equation}
Let $B_i = \{\lambda_{\theta_i}-x^* Ax\geq t + C_1 \sigma r/n\}\cap
B$; then the $1st,\cdots,i-th$ approximate eigenvectors we built are
valid on $B\cap(\cup_{j=1}^i B_j)^c$, so
\begin{align}
&P(\lambda_i (A)\leq\lambda_{\theta_i}-t-C_1 \sigma r/n)\\
&\leq P(\ (B\cap(\cup_{j=1}^i B_j)^c)^c\ )\leq P(B^c)
+ \sum_{j=1}^i P(B_j)\notag\\
&\leq e^{-\frac{m(\theta_i - \sigma)^4}{16\sigma^2 \theta_i^2}} +
8i\cdot e^{-\frac{C_2 m(\theta_i - \sigma)^5 t^2}{\sigma^4 (\theta_i
- \sigma)^3}}\notag
\end{align}

\section{Spiked Population Model:  Proof of Theorem \ref{SPM_Deviation} and Theorem \ref{SPM_Deviation_2}}

Since the distribution of $G$ is orthogonal invariant, we can assume
\begin{equation*}
\Sigma^{\frac{1}{2}} = \diag\{\theta_1,\cdots,\theta_{r+s},1\cdots,1\}
\end{equation*}
The proof is divided into three subsections, corresponding to part
(i) of theorem \ref{SPM_Deviation}, part (i) of theorem
\ref{SPM_Deviation_2}, and part (ii) of both theorems. As mentioned
before, the proof of part (i) of both theorems follows the same idea
used in proving part (i) of theorem \ref{GOE_Deviation}; the proof
of part (ii) of both theorems is similar to the proof of part (ii)
of theorem \ref{GOE_Deviation}.

\subsection{Upper Tail Bound for the Largest Eigenvalues}

We prove part (i) of theorem \ref{SPM_Deviation} in this section.

The $r > 0$ case. By the minimax characterization of eigenvalues, for $1\leq i\leq r$, we have
\begin{align}\label{SPM_Control_By_Norm_1}
\lambda_i (S_n) &= \min_{x^1,\cdots,x^{i-1}\in\mathbb{R}^p} \max_{x\in S^{p-1}\cap\{x^1,\cdots,x^{i-1}\}^{\perp}} x^*S_n x\notag\\
&\leq \max_{x\in S^{p-1},x_1 = \cdots = x_{i-1} = 0}x^*S_n x = \sigma_i^2
\end{align}
$\sigma_i$ is the largest singular value of the lower $(p-i+1)\times
n$ submatrix of $\frac{1}{\sqrt{n}}\Sigma^{\frac{1}{2}}G$. Let
$X_{x,y} = x^*  (  \frac{1}{\sqrt{n}}  \Sigma^{\frac{1}{2}}G)  y,
x\in S^{p-1},y\in S^{n-1}$, then
\begin{equation*}
\sigma_i = \max\{X_{x,y}: x\in S^{p-1}\cap V_i, y\in S^{n-1}\}
\end{equation*}
where $V_i = \{x\in\mathbb{R}^p: x_1 = \cdots = x_{i-1} = 0\}$. Using lemma \ref{Epsilon_Net_Lemma_Singularvalue}, we have
\begin{equation}\label{SPM_Epsilon_Discretization_1}
\sigma_i\leq\frac{1}{1-\epsilon}\max\{X_{x,y}: x\in\mathscr{X}_i,y\in S^{n-1}\},\quad 0<\epsilon<1
\end{equation}
When $i=r$, $\mathscr{X}_r$ is an $\epsilon$-net for $(S^{p-1}\cap
V_r\cap\{x\in\mathbb{R}^p: x_r\geq0\},|\cdot|)$. When $1\leq i\leq
r-1$, $\mathscr{X}_i$ is an $\epsilon$-net for $(S^{p-1}\cap
V_i,|\cdot|)$.

By lemma \ref{Epsilon_Net_Construction_1} and lemma \ref{Epsilon_Net_Construction_2}, we can pick
\begin{equation}\label{SPM_Epsilon_Net_Structure_1}
\mathscr{X}_r = \cup_{u\in\mathscr{N}_r}\mathscr{E}_r (u)
\end{equation}
where $\mathscr{E}_r (u) = \{x\in S^{p-1}: x_1=\cdots =
x_{r-1}=0,x_r=u\},u\in [0,1]$, $\mathscr{N}_r$ is a finite subset of
$[0,1]$ with $|\mathscr{N}_r|\leq\frac{2}{\epsilon}$. When $1\leq
i\leq r-1$
\begin{equation}\label{SPM_Epsilon_Net_Structure_2}
\mathscr{X}_i = \cup_{u\in\mathscr{N}_i}\mathscr{E}_i (u)
\end{equation}
where $\mathscr{E}_i (u) = \{x\in S^{p-1}: x_1=\cdots = x_{i-1} = 0,
(x_i,\cdots,x_r)=u\},u \in B^{r-i+1}_2$, $\mathscr{N}_i$ is a finite
subset of $B^{r-i+1}_2$ with
$|\mathscr{N}_i|\leq\frac{4(r-i+1)^2}{\epsilon}(1 +
\frac{2(r-i+1)}{(r-i)\epsilon})^{r-i}$.

Define $L_{i,u} = \max_{x\in\mathscr{E}_i (u),y\in S^{n-1}} X_{x,y}$, then
\begin{equation}\label{SPM_Stratified_Maximization}
\max\{X_{x,y}: x\in\mathscr{X}_i,y\in S^{n-1}\} = \max\{L_{i,u}: u\in\mathscr{N}_i\}
\end{equation}
\eqref{SPM_Control_By_Norm_1} and \eqref{SPM_Epsilon_Discretization_1} and \eqref{SPM_Stratified_Maximization} imply
\begin{equation}\label{SPM_Lambda_By_Epsilon_1}
\sqrt{\lambda_i (S_n)}\leq\frac{1}{1-\epsilon}\max\{L_{i,u}: u\in\mathscr{N}_i\},\quad 0<\epsilon<1
\end{equation}

We will be using \eqref{SPM_Lambda_By_Epsilon_1} to establish an
upper tail bound for $\lambda_i (S_n)$. As in the GOE case, we will
prove a tail bound for each $L_{i,u}$, then take a union bound over
$u\in\mathscr{N}_i$.

Let $u = (u_i,\cdots,u_r)\in B^{r-i+1}_2$, $x\in\mathscr{E}_i (u), y\in S^{n-1}$, consider
\begin{equation}\label{SPM_Yxy_Definition_1}
Y_{x,y} = \frac{1}{\sqrt{n}}\sqrt{\sum_{k=i}^r \theta_k^2 u_k^2 + 1 - |u|^2} \sum_{j=1}^n y_j \omega_j + \frac{1}{\sqrt{n}}\sum_{j = r+1}^p x'_j \beta_j
\end{equation}
in which $\omega_1,\cdots,\omega_n,\beta_{r+1},\cdots,\beta_p$ are
i.i.d. $\mathcal{N}(0,1)$ random variables; and
\begin{equation*}
x' = (x'_{r+1},\cdots,x'_p) = (\theta_{r+1}x_{r+1},\cdots,\theta_{r+s}x_{r+s},x_{r+s+1},\cdots,x_p)
\end{equation*}
Then for $x,\tilde{x}\in\mathscr{E}_i (u),y,\tilde{y}\in S^{n-1}$, we have
{\allowdisplaybreaks\begin{align}\label{SPM_Slepian_Comparison_1}
&E|X_{x,y} - X_{\tilde{x},\tilde{y}}|^2\\
&= \frac{1}{n}E|\sum_{k=i}^r \sum_{j=1}^n (\theta_k u_k y_j - \theta_k u_k \tilde{y}_j)g_{k,j} + \sum_{k = r+1}^p \sum_{j=1}^n (x'_k y_j - \tilde{x}'_k \tilde{y}_j)g_{k,j}|^2\notag\\
&= \frac{1}{n}\sum_{k=i}^r\sum_{j=1}^n \theta_k^2 u_k^2 (y_j - \tilde{y}_j)^2 + \frac{1}{n}\sum_{k=r+1}^p\sum_{j=1}^n ((x'_k)^2 y_j^2 +(\tilde{x}'_k)^2\tilde{y}_j^2 - 2 x'_k\tilde{x}'_k y_j \tilde{y}_j)\notag\\
&= \frac{1}{n}(\sum_{k=i}^r \theta_k^2 u_k^2)|y - \tilde{y}|^2 + \frac{1}{n}(|x'|^2 + |\tilde{x}'|^2) - \frac{2}{n}(x',\tilde{x}')(y,\tilde{y})\notag\\
&= \frac{1}{n}(\sum_{k=i}^r \theta_k^2 u_k^2 + 1 - |u|^2)|y - \tilde{y}|^2 + \frac{1}{n}|x'-\tilde{x}'|^2\notag\\
&\qquad - \frac{2}{n}(1 - |u|^2 - (x',\tilde{x}'))(1 - (y,\tilde{y}))\notag\\
&\leq \frac{1}{n}(\sum_{k=i}^r \theta_k^2 u_k^2 + 1 - |u|^2)|y - \tilde{y}|^2 + \frac{1}{n}|x'-\tilde{x}'|^2
= E|Y_{x,y} - Y_{\tilde{x},\tilde{y}}|^2\notag
\end{align}}
The maximum of $\{Y_{x,y}: x\in\mathscr{E}_i (u),y\in S^{n-1}\}$ is
reached when $x'$ is a multiple of $(\beta_{r+1},\cdots,\beta_{p})$
and $y$ is a multiple of $(\omega_1,\cdots,\omega_n)$. Since
$\theta_{r+j}\leq 1, |x'|\leq |(x_{r+1},\cdots,x_p)| =
\sqrt{1-|u|^2}$, we have
\begin{align*}
\max_{x\in\mathscr{E}_i (u),y\in S^{n-1}} Y_{x,y}=\frac{1}{\sqrt{n}}\sqrt{\sum_{k=i}^r \theta_k^2 u_k^2 + 1 - |u|^2}\sqrt{\sum_{j=1}^n \omega_j^2} + \frac{1}{\sqrt{n}}|x'|\sqrt{\sum_{j=r+1}^p \beta_j^2}&\\
\leq \frac{1}{\sqrt{n}}\sqrt{\sum_{k=i}^r \theta_k^2 u_k^2 + 1 -
|u|^2}\sqrt{\sum_{j=1}^n \omega_j^2} +
\frac{\sqrt{1-|u|^2}}{\sqrt{n}}\sqrt{\sum_{j=r+1}^p \beta_j^2}&
\end{align*}
Using Slepian's lemma, we have
\begin{align*}
E[\max_{x\in\mathscr{E}_i (u),y\in S^{n-1}}X_{x,y}]&\leq E[\max_{x\in\mathscr{E}_i (u),y\in S^{n-1}}Y_{x,y}]\\
&\leq \sqrt{\sum_{k=i}^r \theta_k^2 u_k^2 + 1 - |u|^2} + \sqrt{\frac{p-r}{n}}\sqrt{1 - |u|^2}
\end{align*}
Therefore
\begin{equation*}
E[L_{i,u}]\leq  \sqrt{\sum_{k=i}^r \theta_k^2 u_k^2 + 1 - |u|^2} + \sqrt{c}\sqrt{1 - |u|^2} = \varphi(u),\quad c=\frac{p-r}{n}
\end{equation*}

When $\theta_i^2 > 1 + \sqrt{c}$, the maximum of $\varphi(u)$ is $\sqrt{\theta_i^2 + c\cdot\frac{\theta_i^2}{\theta_i^2 - 1}}$ and is attained when
\begin{equation}\label{SPM_Optimal_Position_1}
u = (\sqrt{\frac{(\theta_i^2 - 1)^2 - c}{(\theta_i^2 - 1)(\theta_i^2 - 1 +c)}},0,\cdots,0)
\end{equation}
When $\theta_i^2 \leq 1 + \sqrt{c}$, the maximum of $\varphi(u)$ is $\varphi(0) = 1 + \sqrt{c}$. Hence
\begin{equation}\label{SPM_Lu_Expectation_1}
E[L_{i,u}]\leq\sqrt{\lambda_{\theta_i,c}}
\end{equation}

Since $E[X_{x,y}^2] = \frac{1}{n}(\sum_{k=i}^r \theta_k^2 u_k^2 +
|x'|^2)\leq \frac{1}{n}\theta_i^2$, we can use
\eqref{SPM_Lu_Expectation_1} and apply proposition
\ref{Minimax_Gaussian_Concentration} to get
\begin{equation*}
P(L_{i,u}\geq \sqrt{\lambda_{\theta_i, c}} + t)\leq e^{-\frac{nt^2}{2\theta_i^2}},\quad t\geq 0
\end{equation*}
Let $\epsilon = \frac{(1-a)t}{\sqrt{\lambda_{\theta_i, c}} + t},
0<a<1$, then $(1-\epsilon) (\sqrt{\lambda_{\theta_i,c}} + t) =
\sqrt{\lambda_{\theta_i,c}} + at$. Using
\eqref{SPM_Lambda_By_Epsilon_1} and the above inequality, we have
\begin{align*}
P(\sqrt{\lambda_i (S_n)}\geq \sqrt{\lambda_{\theta_i,c}} + t)&\leq P(\max\{L_{i,u}: u\in\mathscr{N}_i\}\geq (1-\epsilon)(\sqrt{\lambda_{\theta_i,c}} + t))\\
&\leq \sum_{u\in\mathscr{N}_i}P(L_{i,u}\geq \sqrt{\lambda_{\theta_i,c}} + at)\\
&\leq |\mathscr{N}_i|e^{-\frac{na^2 t^2}{2\theta_i^2}}
\end{align*}
The final step is to use our bound on $|\mathscr{N}_i|$ and optimize
over $a\in (0,1)$. When
$t\geq\frac{\sqrt{r-i+1}\theta_i}{\sqrt{\delta (1-\delta)n}}$,
choose $a = \frac{1}{2}(1 + \sqrt{1 - \frac{4(r-i+1)\theta_i^2
}{nt^2}})$, then $a\geq 1-\delta\geq\frac{2}{3}$, $0 <\epsilon \leq
\frac{1}{3}$ (as needed in the use of
\eqref{SPM_Lambda_By_Epsilon_1} and lemma
\ref{Epsilon_Net_Construction_1}). When $i = r$, we use
$|\mathscr{N}_r|\leq\frac{2}{\epsilon}$; when $1\leq i\leq r-1$, we
use $|\mathscr{N}_i|\leq \frac{4(r-i+1)^2}{\epsilon}(1 +
\frac{2(r-i+1)}{(r-i)\epsilon})^{r-i}$. Then
\begin{equation*}
P(\sqrt{\lambda_i (S_n)}\geq \sqrt{\lambda_{\theta_i,c}} + t) \leq C_{r-i+1,\theta_i}(n)\cdot e^{-\frac{na^2 t^2}{2\theta_i^2}}
\leq C_{r-i+1,\theta_i}(n)\cdot e^{-\frac{(1-\delta)^2 nt^2}{2\theta_i^2}}
\end{equation*}
This finishes the proof of the $r>0$ case.

When $r = 0$, $\sqrt{\lambda_1 (S_n)} = \max_{x\in S^{p-1},y\in
S^{n-1}}X_{x,y}$, $X_{x,y} =
x^*(\frac{1}{\sqrt{n}}\Sigma^{\frac{1}{2}}G)y$. For $x\in S^{p-1},
y\in S^{n-1}$, define
\begin{equation*}
Y_{x,y} = \frac{1}{\sqrt{n}}\sum_{j=1}^n y_j \omega_j + \frac{1}{\sqrt{n}}\sum_{k=1}^p x_k' \beta_k
\end{equation*}
where $\omega_1,\cdots,\omega_n,\beta_1,\cdots,\beta_p$ are i.i.d. $\mathcal{N}(0,1)$ random variables and
\begin{equation*}
x' = (x_1',\cdots,x_p') = (\theta_1 x_1,\cdots,\theta_s x_s,x_{s+1},\cdots,x_p)
\end{equation*}
Similar to \eqref{SPM_Slepian_Comparison_1}, we have $E|X_{x,y} -
X_{\tilde{x},\tilde{y}}|^2\leq E|Y_{x,y} -
Y_{\tilde{x},\tilde{y}}|^2, \forall x,\tilde{x}\in S^{p-1}$,
$y,\tilde{y}\in S^{n-1}$. Thus $E[\sqrt{\lambda_1 (S_n)}]\leq
E[\max_{x\in S^{p-1},y\in S^{n-1}}Y_{x,y}]\leq 1 + \sqrt{c}$ by
Slepian's lemma. Using proposition
\ref{Minimax_Gaussian_Concentration}, we have
\begin{equation*}
P(\sqrt{\lambda_1 (S_n)}\geq 1 + \sqrt{c} + t)\leq e^{-\frac{nt^2}{2}},\quad t\geq 0
\end{equation*}

\subsection{Lower Tail Bound for the Smallest Eigenvalues}

We prove part (i) of theorem \ref{SPM_Deviation_2} in this section.

The $s>0$ case. By the max-min characterization of eigenvalues, for
$1\leq i\leq s$, we have
\begin{align}\label{SPM_Control_By_Norm_2}
\lambda_{p-i+1}(S_n) &= \max_{V\subset\mathbb{R}^p,\dim V = p-i+1} \min_{x\in S^{p-1}\cap V} x^*S_n x\notag\\
&\geq \min_{x\in S^{p-1}\cap W_i} x^*S_n x = \mu_i^2
\end{align}
where $W_i = \{x\in\mathbb{R}^p: x_{r+s-i+2} = \cdots = x_{r+s} = 0\}$ and
\begin{equation*}
\mu_i = \min_{x\in S^{p-1}\cap W_i}\max_{y\in S^{n-1}} X_{x,y},\quad X_{x,y} = x^*(\frac{1}{\sqrt{n}}\Sigma^{\frac{1}{2}}G)y
\end{equation*}
Using lemma \ref{Epsilon_Net_Lemma_Smallest_Singularvalue}, we have
\begin{equation}\label{SPM_Epsilon_Discretization_2}
\min_{x\in\mathscr{X}_i}\max_{y\in S^{n-1}} X_{x,y} \leq \mu_i + \epsilon\sqrt{\lambda_{1}(S_n)}
\end{equation}
When $r=0,i=s$, $\mathscr{X}_s$ is an $\epsilon$-net for
$(S^{p-1}\cap W_s\cap\{x\in\mathbb{R}^p: x_1\geq 0 \},|\cdot|)$. In
all other situations, $\mathscr{X}_i$ is an $\epsilon$-net for
$(S^{p-1}\cap W_i, |\cdot|)$.

By lemma \ref{Epsilon_Net_Construction_1} and lemma \ref{Epsilon_Net_Construction_2}, when $r=0,i=s$, we can arrange
\begin{equation}\label{SPM_Epsilon_Net_Structure_3}
\mathscr{X}_s = \cup_{u\in \mathscr{N}_s}\mathscr{E}_s (u)
\end{equation}
where $\mathscr{E}_s (u) = \{x\in\mathbb{R}^p: x_2 = \cdots = x_s =
0, x_1 = u\}, u\in [0,1]$, $\mathscr{N}_s$ is a finite subset of
$[0,1]$ with $|\mathscr{N}_s|\leq\frac{2}{\epsilon}$; in all other
situations
\begin{equation}\label{SPM_Epsilon_Net_Structure_4}
\mathscr{X}_i =\cup_{u\in\mathscr{N}_i}\mathscr{E}_i (u)
\end{equation}
where for $u\in B^{r+s-i+1}_2$
\begin{equation*}
\mathscr{E}_i (u) = \{x\in\mathbb{R}^p: x_{r+s-i+2} = \cdots = x_{r+s} = 0, (x_1,\cdots,x_{r+s-i+1}) = u\}
\end{equation*}
$\mathscr{N}_i$ is a finite subset of $B^{r+s-i+1}_2$ with
$|\mathscr{N}_i|\leq \frac{4(r+s-i+1)^2}{\epsilon}(1 +
\frac{2(r+s-i+1)}{(r+s-i)\epsilon})^{r+s-i}$.

Define $L_{i,u} = \min_{x\in\mathscr{E}_i (u)}\max_{y\in S^{n-1}} X_{x,y}$, then
\begin{equation}\label{SPM_Stratified_Minimization}
\min_{x\in\mathscr{X}_i}\max_{y\in S^{n-1}} X_{x,y} = \min_{u\in\mathscr{N}_i} L_{i,u}
\end{equation}
\eqref{SPM_Control_By_Norm_2}, \eqref{SPM_Epsilon_Discretization_2} and \eqref{SPM_Stratified_Minimization} imply
\begin{equation}\label{SPM_Lambda_By_Epsilon_2}
\min_{u\in\mathscr{N}_i}L_{i,u}\leq \sqrt{\lambda_{p-i+1}(S_n)} + \epsilon\sqrt{\lambda_1 (S_n)}
\end{equation}
To get a lower tail bound for $\lambda_{p-i+1}(S_n)$, we will
establish a lower tail bound for each $L_{i,u}$ then take an union
bound. The $\sqrt{\lambda_1 (S_n)}$ term will be dealt with using
the result of part (i) of theorem \ref{SPM_Deviation}.

For $x\in\mathscr{E}_i (u),y\in S^{n-1}$, consider
\begin{equation*}
Y_{x,y} = \frac{1}{\sqrt{n}}\sqrt{\sum_{k=1}^{r+s-i+1}\theta_k^2 u_k^2 + 1- |u|^2}\sum_{j=1}^n y_j \omega_j + \frac{1}{\sqrt{n}}\sum_{k=r+s+1}^p x_k\beta_k
\end{equation*}
in which $\omega_1,\cdots,\omega_n,\beta_{r+s+1},\cdots,\beta_p$ are
i.i.d. $\mathcal{N}(0,1)$ random variables. Similar to
\eqref{SPM_Slepian_Comparison_1}, for $x,\tilde{x}\in\mathscr{E}_i
(u), y,\tilde{y}\in S^{n-1}$, we have
\begin{align}\label{SPM_Slepian_Comparison_2}
&E|Y_{x,y} -Y_{\tilde{x},\tilde{y}}|^2 - E|X_{x,y} - X_{\tilde{x},\tilde{y}}|^2 \notag\\
&\quad= \frac{2}{n}(1-|u|^2 - \sum_{k=r+s+1}^p x_k\tilde{x}_k)(1 - (y,\tilde{y}))
\end{align}
This quantity is always non-negative, and equals zero when $x =
\tilde{x}$. To proceed, we will use the following generalization of
Slepian's lemma.
\begin{proposition}\emph{\cite{gordon85}}\label{Gordon_1}
Let $(X_{i,j})_{1\leq i\leq n,1\leq j\leq m}$ and $(Y_{i,j})_{1\leq
i \leq n,1\leq j\leq m}$ be two centered Gaussian processes. Assume
\begin{align*}
&E|X_{i,j} - X_{i,k}|^2 \leq E|Y_{i,j} - Y_{i,k}|^2\\
&E|X_{i,j} - X_{l,k}|^2 \geq E|Y_{i,j} - Y_{l,k}|^2, \quad i\neq l
\end{align*}
Then
\begin{equation*}
E[\min_{i}\max_{j}X_{i,j}]\leq E[\min_{i}\max_{j}Y_{i,j}]
\end{equation*}
\end{proposition}

This proposition implies
\begin{align*}
&E[\min_{x\in\mathscr{E}_i (u)}\max_{y\in S^{n-1}}X_{x,y}]\geq E[\min_{x\in\mathscr{E}_i (u)}\max_{y\in S^{n-1}}Y_{x,y}]\\
& = \frac{1}{\sqrt{n}}\sqrt{\sum_{k=1}^{r+s-i+1}\theta_k^2 u_k^2 + 1- |u|^2}E[\sqrt{\sum_{j=1}^n \omega_j^2}] -\frac{\sqrt{1 - |u|^2}}{\sqrt{n}}E[\sqrt{\sum_{k=r+s+1}^p \beta_k^2}]\\
&\geq \frac{1}{\sqrt{n}}\sqrt{\sum_{k=1}^{r+s-i+1}\theta_k^2 u_k^2 + 1- |u|^2}\frac{\sqrt{2}\Gamma(\frac{n+1}{2})}{\Gamma(\frac{n}{2})} - \frac{\sqrt{p-r-s}}{\sqrt{n}}\sqrt{1-|u|^2}\\
&\geq (1-\frac{1}{2n})\sqrt{\sum_{k=1}^{r+s-i+1}\theta_k^2 u_k^2 + 1- |u|^2} - \sqrt{c'}\sqrt{1-|u|^2},\quad c' = \frac{p-r-s}{n}
\end{align*}
The last step uses Stirling's approximation for $\Gamma$-functions.
Therefore
\begin{equation*}
E[L_{i,u}]\geq \sqrt{\sum_{k=1}^{r+s-i+1}\theta_k^2 u_k^2 + 1- |u|^2} - \sqrt{c'}\sqrt{1-|u|^2} - \frac{\theta_1\vee 1}{2n} = \varphi(u) - \frac{\theta_1\vee 1}{2n}
\end{equation*}

To avoid heavy notation, let $\hat{\theta}_i = \theta_{r+s-i+1},
1\leq i\leq s$, i.e. $\hat{\theta}_i^2$ is the $i$-th smallest
eigenvalue of the population covariance matrix $\Sigma$. When
$\hat{\theta}_i^2 < 1 - \sqrt{c'}$, the minimum of $\varphi(u)$ is
$\sqrt{\hat{\theta}_i^2 +
c'\cdot\frac{\hat{\theta}_i^2}{\hat{\theta}_i^2 - 1}}$ and is
attained when
\begin{equation}\label{SPM_Optimal_Position_2}
u = (0,\cdots,0,\sqrt{\frac{(\hat{\theta}_i^2 - 1)^2 - c'}{(\hat{\theta}_i^2 - 1)(\hat{\theta}_i^2 - 1 + c')}})
\end{equation}
When $\hat{\theta}_i^2 \geq 1 - \sqrt{c'}$, the minimum of
$\varphi(u)$ is $\varphi(0) = 1 - \sqrt{c'}$. Hence
\begin{equation}\label{SPM_Lu_Expectation_2}
E[L_{i,u}] \geq \sqrt{\lambda_{\hat{\theta}_i, c'}} - \frac{\theta_1\vee 1}{2n}
\end{equation}

Since $E[X_{x,y}^2] = \frac{1}{n}(\sum_{k=1}^{r+s-i+1}\theta_k^2
u_k^2 + 1 - |u|^2)\leq \frac{\theta_1^2\vee 1}{n}$, we can use
\eqref{SPM_Lu_Expectation_2} and apply proposition
\ref{Minimax_Gaussian_Concentration} to get
\begin{equation*}
P(L_{i,u}\leq \sqrt{\lambda_{\hat{\theta}_i,c'}} - \frac{\theta_1\vee 1}{2n} - t)\leq e^{\frac{-nt^2}{2(\theta_1^2\vee 1)}}
\end{equation*}
Let $(1-a)t = \epsilon \tilde{t},0<a<1$,
\eqref{SPM_Lambda_By_Epsilon_2} and the above inequality imply
\begin{align*}
&P(\sqrt{\lambda_{p-i+1}(S_n)}\leq
\sqrt{\lambda_{\hat{\theta}_i,c'}} - \frac{\theta_1\vee 1}{2n} - t)\\
&\quad\leq P(\min_{u\in\mathscr{N}_i}L_{i,u} \leq
\sqrt{\lambda_{\hat{\theta}_i, c'}} - \frac{\theta_1\vee 1}{2n} -
at) + P(\sqrt{\lambda_1 (S_n)}\geq \tilde{t})\\
&\quad\leq\sum_{u\in\mathscr{N}_i}
P(L_{i,u}\leq\sqrt{\lambda_{\hat{\theta}_i, c'}} -
\frac{\theta_1\vee 1}{2n} - at) +P(\sqrt{\lambda_1 (S_n)}\geq
\tilde{t})\\
&\quad\leq |\mathscr{N}_i|e^{-\frac{na^2 t^2}{2(\theta_1^2\vee 1)}}
+ P(\sqrt{\lambda_1 (S_n)}\geq \tilde{t})
\end{align*}
We will use the bound for $|\mathscr{N}_i|$ and make an appropriate
choice of $a$ and $\tilde{t}$. When $r>0$, we pick $\tilde{t} =
\sqrt{\lambda_{\theta_1,c}}+t$, $a=\frac{1}{2}(1 + \sqrt{1 -
\frac{4(r+s-i+1)\theta_1^2}{nt^2}})$; when $r=0$, we pick $\tilde{t}
= 1 + \sqrt{c} + t$, $a = \frac{1}{2}(1 + \sqrt{1 -
\frac{4(r+s-i+1)}{nt^2}})$. This gives
\begin{equation*}
P(\sqrt{\lambda_{p-i+1}(S_n)}\leq \sqrt{\lambda_{\hat{\theta}_i,c'}}
- \frac{\theta_1\vee 1}{2n} - t)\leq C'_i (n) e^{-\frac{(1-\delta)^2
nt^2}{2(\theta_1^2\vee 1)}}
\end{equation*}
This finishes the proof of the $s>0$ case.

When $s=0$, we have
\begin{equation*}
\sqrt{\lambda_p (S_n)}=\min_{x\in S^{p-1}}\max_{y\in
S^{n-1}}X_{x,y},\quad X_{x,y} = x^*
(\frac{1}{\sqrt{n}}\Sigma^{\frac{1}{2}}G)y
\end{equation*}
For $u\in B^r_2$, let $\mathscr{E}(u)=\{x\in S^{p-1}:
(x_1,\cdots,x_r)=u\}$. For $x\in\mathscr{E}(u),y\in S^{n-1}$, define
\begin{equation*}
Y_{x,y}=\frac{1}{\sqrt{n}}\sqrt{\sum_{k=1}^r \theta_k^2 u_k^2 + 1 - |u|^2}\sum_{j=1}^n y_j \omega_j + \frac{1}{\sqrt{n}}\sum_{k=r+1}^p x_k \beta_k
\end{equation*}
where $\omega_1,\cdots,\omega_n,\beta_{r+1},\cdots,\beta_p$ are
i.i.d. $\mathcal{N}(0,1)$ random variables. Similar to
\eqref{SPM_Slepian_Comparison_1}, for
$x,\tilde{x}\in\mathscr{E}(u),y,\tilde{y}\in S^{n-1}$, we have
\begin{equation*}
E|Y_{x,y} - Y_{\tilde{x},\tilde{y}}|^2 - E|X_{x,y} - X_{\tilde{x},\tilde{y}}|^2 = \frac{2}{n}(1-|u|^2 - \sum_{k=r+1}^p x_k \tilde{x}_k)(1-(y,\tilde{y}))
\end{equation*}
This is nonnegative, and equals zero when $x=\tilde{x}$. Using proposition \ref{Gordon_1}, when $r=0$, we have
\begin{equation*}
E[\min_{x\in S^{p-1}}\max_{y\in S^{n-1}}X_{x,y}]\geq E[\min_{x\in S^{p-1}}\max_{y\in S^{n-1}}Y_{x,y}]\geq 1 - \sqrt{c'}
\end{equation*}
when $r>0$, we have
\begin{equation*}
E[\min_{x\in\mathscr{E}(u)}\max_{y\in S^{n-1}}X_{x,y}]\geq E[\min_{x\in\mathscr{E}(u)}\max_{y\in S^{n-1}}Y_{x,y}]
\geq 1 - \sqrt{c'} -\frac{\theta_1}{2n}
\end{equation*}
The derivation is similar to the $s>0$ case.

Let $L_u = \min_{x\in\mathscr{E}(u)}\max_{y\in S^{n-1}}X_{x,y}$, the
previous two inequalities and proposition
\ref{Minimax_Gaussian_Concentration} imply
\begin{equation}\label{SPM_RS_Zero_Bound}
P(\sqrt{\lambda_p (S_n)}\leq 1 - \sqrt{c'} - t)\leq e^{-\frac{nt^2}{2}},\quad t\geq 0, r=0
\end{equation}
This is the result for the $r=s=0$ case; and
\begin{align*}
P(L_u \leq 1 - \sqrt{c'} - \frac{\theta_1}{2n}-t)\leq e^{-\frac{nt^2}{2\theta_1^2}},\quad t\geq 0, r>0
\end{align*}

When $r>0$, we can use an $\epsilon$-net argument analogous to the one used in establishing \eqref{SPM_Lambda_By_Epsilon_2} to get
\begin{equation*}
\min_{u\in\mathscr{N}}L_u \leq\sqrt{\lambda_p (S_n)} + \epsilon\sqrt{\lambda_1 (S_n)}
\end{equation*}
$\mathscr{N}$ is a finite set with
$|\mathscr{N}|\leq\frac{2}{\epsilon}$ when $r=1$; and
$|\mathscr{N}|\leq\frac{4r^2}{\epsilon}(1+\frac{2(r-1)}{r\epsilon})^{r-1}$
when $r>1$.

Let $(1-a)t=\epsilon (\sqrt{\lambda_{\theta_1,c}}+t),0<a<1$. Using
the previous two inequalities and the result of part (i) of theorem
\ref{SPM_Deviation}, we have
\begin{align*}
&P(\sqrt{\lambda_p (S_n)}\leq 1 - \sqrt{c'} - \frac{\theta_1}{2n}-t)\\
&\quad\leq P(\min_{u\in\mathscr{N}}L_u \leq 1 - \sqrt{c'} - \frac{\theta_1}{2n} -at) + P(\sqrt{\lambda_1 (S_n)}\geq \sqrt{\lambda_{\theta_1,c}}+t)\\
&\quad \leq |\mathscr{N}|e^{-\frac{na^2 t^2}{2\theta_1^2}} + C_{r,\theta_1}(n)e^{-\frac{(1-\delta)^2 nt^2}{2\theta_1^2}}
\end{align*}
Using the bound for $|\mathscr{N}|$ and choosing $a=\frac{1}{2}(1+\sqrt{1 - \frac{4r\theta_1^2}{nt^2}})$ gives
\begin{equation*}
P(\sqrt{\lambda_p (S_n)}\leq 1 - \sqrt{c'} - \frac{\theta_1}{2n}-t)\leq 2 C_{r,\theta_1}(n)e^{-\frac{(1-\delta)^2 nt^2}{2\theta_1^2}}
\end{equation*}
This together with \eqref{SPM_RS_Zero_Bound} finishes the proof of the $s=0$ case.

\subsection{Approximate Eigenvectors}
We prove part (ii) of both theorem \ref{SPM_Deviation} and
\ref{SPM_Deviation_2} in this section. The proof uses the same
method for proving part (ii) of theorem \ref{GOE_Deviation}. So we
will focus on the construction of approximate eigenvectors and omit
other details.

Let's first consider approximate eigenvectors associated with the
largest eigenvalues.  Let $\tilde{G}$ be the lower $(p-r-s)\times n$
submatrix of $\frac{1}{\sqrt{n}}G$. Let $(1+\sqrt{c})^2 < \lambda_0
< \lambda_{\theta_i,c}$,
$B=\{\lambda_{max}(\tilde{G}^*\tilde{G})\leq \lambda_0\}$, then
\eqref{SPM_Bound_1_Null_Case} in (i) implies
\begin{equation}\label{SPM_Truncation_Error_Prob}
P(B^c) \leq e^{-\frac{1}{2}n(\sqrt{\lambda_0} - 1 - \sqrt{c})^2}
\end{equation}
We will construct approximate eigenvectors on the event $B$.

Let $x\in S^{p-1}$ be such that $x_1=\cdots =
x_{i-1}=x_{i+1}=\cdots= x_{r+s}=0$, $x_i= \sqrt{\frac{(\theta_i^2 -
1)^2 - c}{(\theta_i^2 - 1)(\theta_i^2 - 1 + c)}}$, and
\begin{align}\label{SPM_Schur_Complement_Eigenvector}
&\tilde{x} = (x_{r+s+1},\cdots,x_p)^t = -
\sqrt{1-x_i^2}\frac{R\tilde{G}v}{\sqrt{\lambda L_2 +
L_1}}\\
& v = \frac{1}{\sqrt{n}}(g_{i,1},\cdots,g_{i,n})^t\notag
\end{align}
where $\lambda=\lambda_{\theta_i,c}$, $R = (\tilde{G}\tilde{G}^* -
\lambda I)^{-1}$, $S = (\tilde{G}^* \tilde{G} - \lambda I)^{-1}$,
$L_1 = v^* S v, L_2 = v^* S^2 v$. Then $x^* S_n x = |y|^2$ with
\begin{align*}
y&=(\frac{1}{\sqrt{n}}G^*)\Sigma^{\frac{1}{2}}x = \theta_i x_i v +
\tilde{G}^*\tilde{x}\\
&=\theta_i x_i v - \frac{\sqrt{1-x_i^2}}{\sqrt{\lambda L_2 +
L_1}}\tilde{G}^* R \tilde{G}v\\
&=\theta_i x_i v - \frac{\sqrt{1-x_i^2}}{\sqrt{\lambda L_2 + L_1}}
(I +
 \lambda S)v\\
&= (aI - \lambda bS)v
\end{align*}
where $a = \theta_i x_i - b, b = \frac{\sqrt{1-x_i^2}}{\sqrt{\lambda
L_2 + L_1}}$. Therefore
\begin{equation}\label{SPM_Lambda_Minus_xSx}
\lambda - x^*S_n x = \lambda (1 - \lambda b^2 L_2) - a^2 |v|^2 +
2ab\lambda L_1
\end{equation}

To build a deviation inequality for $\lambda - x^* S_n x$, we will
prove $1 - \lambda b^2 L_2 \approx 0$, $a\approx 0$, then take a
union bound in \eqref{SPM_Lambda_Minus_xSx}. This is accomplished by
proving $L_1\approx g(\lambda)$, $L_2\approx g'(\lambda)$, where
$g(z)$ is the Stieltjes transform of the Marcenko-Pastur
distribution, see \cite{mp67}. More precisely, as
$n\rightarrow\infty$ and holding the ratio $p/n$ constant, the
spectral distribution of $\tilde{G}^*\tilde{G}$ converges to a
deterministic limiting distribution with Stieltjes transform
\begin{equation*}
g(z) = \frac{c-1-z + \sqrt{(z-1-c)^2 - 4c}}{2z}
\end{equation*}
Thus we can apply the same argument used in proving part (ii) of
theorem \ref{GOE_Deviation} to get
\begin{equation}\label{SPM_Stieltjes_At_Lambda_Theta}
L_1\approx g(\lambda) = -\frac{1}{\theta_i^2},\quad L_2 \approx
g'(\lambda) = \frac{(\theta_i^2 - 1)^2}{\theta_i^4 ((\theta_i^2 -
1)^2 - c)}
\end{equation}
The only modification is that we need proposition
\ref{MP_Convergence_Rate}, which is a rate of convergence result for
sample covariance type matrices, instead of proposition
\ref{GOE_Convergence_Rate}.

The above is the proof of part (ii) of theorem \ref{SPM_Deviation}.
The proof of part (ii) of theorem \ref{SPM_Deviation_2} is similar.
In this case, we will build approximate eigenvectors on
$B=\{\lambda_{\min}(\tilde{G}^* \tilde{G})\geq \lambda_0\}$,
$\lambda_{\theta_{r+s-i+1},c'} < \lambda_0 < (1 - \sqrt{c'})^2$, and
use \eqref{SPM_Bound_2_Null_Case_1} and
\eqref{SPM_Bound_2_Null_Case_2} to get a bound for $P(B^c)$ similar
to \eqref{SPM_Truncation_Error_Prob}. The construction of
approximate eigenvectors is done by changing $i$ to $r+s-i+1$ in the
above argument.

\section{Summary}
In this work we considered the extreme eigenvalues of matrices from
the deformed GOE and the spiked population model. We proved tight
deviation bounds for these eigenvalues. An interesting direction to
go next is to study these problems when the Gaussian distribution is
replaced by a stable distribution. This will complete our picture
about the eigenvalues of deformed random matrices.

\section{Appendix}

This appendix collects some auxiliary propositions.

\begin{lemma}\label{Epsilon_Net_Lemma_Eigenvalue}
Let $A\in\mathbb{R}^{n\times n}_{sym}$, $0 < \epsilon < \frac{1}{2}$, $\mathscr{X}$ is an $\epsilon$-net for $(S^{n-1},|\cdot|)$, then
\begin{equation*}
\|A\|\leq \frac{1}{1-2\epsilon}\max_{u\in\mathscr{X}}|u^* A u|
\end{equation*}
The same inequality holds if $\mathscr{X}$ is an $\epsilon$-net for $(S^{n-1}_{+},|\cdot|)$ where $S^{n-1}_{+} = \{x\in S^{n-1}:  x_1\geq 0\}$.
\end{lemma}

\begin{proof}
Let $x\in S^{n-1}$ be such that $\|A\| = x^* A x$. We can arrange
$x\in S^{n-1}_{+}$ (by changing $x$ to $-x$ if $x_1 <0$) when
$\mathscr{X}$ is an $\epsilon$-net for $S^{n-1}_+$. Choose
$y\in\mathscr{X}$ which approximates $x$ as $|x-y|\leq\epsilon$. By
the triangle inequality, we have
\begin{equation*}
|x^*Ax -y^*Ay|=|x^*A(x-y) + (x-y)^*Ay|\leq2\epsilon\|A\|
\end{equation*}
It follows that $|y^* Ay|\geq |x^*Ax|-2\epsilon\|A\| = \|A\| - 2\epsilon \|A\|$, $\|A\|\leq\frac{1}{1-2\epsilon}|y^*Ay|\leq\frac{1}{1-2\epsilon}\max_{u\in\mathscr{X}}|u^*Au|$.
\end{proof}

\begin{lemma}\label{Epsilon_Net_Lemma_Singularvalue}
Let $A\in\mathbb{R}^{p\times n}$, $0 < \epsilon <1$, $\mathscr{X}$ is an $\epsilon$-net for $(S^{p-1},|\cdot|)$, then
\begin{equation*}
\|A\|\leq\frac{1}{1-\epsilon}\max_{x\in\mathscr{X},y\in S^{n-1}} x^*Ay
\end{equation*}
The same inequality holds if $\mathscr{X}$ is an $\epsilon$-net for $(S^{p-1}_+,|\cdot|)$ where $S^{p-1}_+ = \{x\in S^{p-1}: x_1\geq 0\}$.
\end{lemma}

\begin{proof}
Pick $x_0\in S^{p-1},y_0\in S^{n-1}$ so that $\|A\| = x_0^* A y_0$.
We can arrange $x_0\in S^{p-1}_+$ in the $S^{p-1}_+$ case. Find
$\tilde{x}\in\mathscr{X}$ with $|\tilde{x} - x_0|\leq\epsilon$, then
\begin{equation*}
|\tilde{x}^*Ay_0 - x_0^*Ay_0|\leq \|A\|\cdot|\tilde{x}-x_0|\cdot|y_0|\leq\epsilon\|A\|
\end{equation*}
Thus
\begin{align*}
&\tilde{x}^*Ay_0\geq x_0^*Ay_0 - \epsilon \|A\| = \|A\| - \epsilon\|A\|\\
&\|A\|\leq \frac{1}{1-\epsilon}\tilde{x}^*Ay_0\leq\frac{1}{1-\epsilon}\max_{x\in\mathscr{X},y\in S^{n-1}} x^*Ay
\end{align*}
\end{proof}

\begin{lemma}\label{Epsilon_Net_Lemma_Smallest_Singularvalue}
Let $A\in\mathbb{R}^{p\times n},p < n$, $s_{\min}(A)$ is the smallest singular value of $A$. Let $\mathscr{X}$ be an $\epsilon$-net of $(S^{p-1},|\cdot|)$, then
\begin{equation*}
\min_{x\in\mathscr{X}}\max_{y\in S^{n-1}}x^*Ay \leq s_{\min}(A) + \epsilon \|A\|
\end{equation*}
The same inequality holds if $\mathscr{X}$ is an $\epsilon$-net of $(S^{p-1}_+,|\cdot|)$ where $S^{p-1}_+ = \{x\in S^{p-1}: x_1\geq 0\}$.
\end{lemma}

\begin{proof}
Let  $s_{\min}(A) = x_0^*Ay_0 = |A^* x_0|, x_0\in S^{p-1}, y_0\in
S^{n-1}$, we can arrange $x_0\in S^{p-1}_+$ when $\mathscr{X}$ is an
$\epsilon$-net for $S^{p-1}_+$. Find $\tilde{x}\in\mathscr{X}$ so
that $|\tilde{x} - x_0|\leq \epsilon$, then
\begin{align*}
&\min_{x\in\mathscr{X}}\max_{y\in S^{n-1}}x^*Ay\leq\max_{y\in S^{n-1}}\tilde{x}^*Ay = |A^*\tilde{x}|\\
&\quad\leq |A^*x_0| + |A^*(\tilde{x} - x_0)|\leq s_{\min}(A) + \epsilon\|A\|
\end{align*}
\end{proof}

\begin{lemma}\label{Epsilon_Net_Construction_1} Let $B^m_2 =\{x\in\mathbb{R}^m: \sum_{i = 1}^{m} x_i^2 \leq 1\}$, for $x,y\in B^m_2$, define
\begin{equation*}
\rho_m (x,y) = \sqrt{|x - y|^2 + (\sqrt{1 - |x|^2} - \sqrt{1 - |y|^2})^2}
\end{equation*}
Then

(i):  $\rho_m$ is a metric on $B^m_2$.

(ii):  For $0 < \epsilon \leq \frac{1}{3}$, there exists an $\epsilon$-net for $([0,1], \rho_1)$ with size $\leq \frac{2}{\epsilon}$.

(iii):  For $0 < \epsilon \leq \frac{1}{3}$, when $m\geq 2$, there
exists an $\epsilon$-net for $(B^m_2, \rho_m)$ with size $\leq
\frac{4m^2}{\epsilon}(1 + \frac{2m}{(m-1)\epsilon})^{m-1}$.
\end{lemma}

\begin{proof}
(i):  To check the triangle inequality, we use $\sqrt{(a+b)(c+d)}\geq\sqrt{ac} + \sqrt{bd}$ to lower bound the cross term in the expansion of $(\rho_m (x,y) + \rho_m (y,z))^2$
\begin{align*}
&(\rho_m (x,y) + \rho_m (y,z))^2\geq |x-y|^2 + (\sqrt{1-|x|^2} - \sqrt{1 - |y|^2})^2\\
 &\quad + |y-z|^2 + (\sqrt{1-|y|^2} - \sqrt{1 - |z|^2})^2\\
 &\quad + 2(|x-y|\cdot|y-z| + |\sqrt{1-|x|^2} - \sqrt{1- |y|^2}|\cdot|\sqrt{1 - |y|^2} - \sqrt{1-|z|^2}|)\\
&\geq (x-y+y-z)^2 + (\sqrt{1-|x|^2} - \sqrt{1- |y|^2}+\sqrt{1 - |y|^2} - \sqrt{1-|z|^2})^2\\
&= \rho_m (x,z)^2
\end{align*}

(ii):  Let $1 > \eta_1 > \eta_1' > \eta_2 > \eta_2' > \cdots \geq 0$ be such that
\begin{align*}
&\rho_1 (x,\eta_1)\leq \epsilon,\quad x\in [\eta_1, 1]\\
&\rho_1 (x,\eta_i) \leq \epsilon, \quad x\in [\eta_i',\eta_i], \forall i\\
&\rho_1 (x,\eta_{i+1}) \leq \epsilon, \quad x\in[\eta_{i+1},\eta_i'], \forall i
\end{align*}
Then $\{\eta_1, \eta_2, \cdots\}$ is an $\epsilon$-net. We can pick
$\eta_1 = 1 -\frac{\epsilon^2}{2}$. For $0\leq u \leq v <1$, by
applying the mean value theorem on $\sqrt{1-u^2} - \sqrt{1 - v^2}$,
we have
\begin{equation*}
\rho_1 (u,v)\leq\sqrt{(u-v)^2 + \frac{v^2}{1 - v^2} (u-v)^2} = \frac{v-u}{\sqrt{1-v^2}}
\end{equation*}
Writing $\eta_i = 1 - x_i \epsilon^2, \eta_i' = 1 - x_i' \epsilon^2$ and using the above inequality, we have
\begin{align*}
&\rho_1 (\eta_i',\eta_i)\leq \frac{\eta_i - \eta_i'}{\sqrt{1 - \eta_i^2}} = \frac{(x_i'- x_i)\epsilon}{\sqrt{x_i( 1 + \eta_i)}}\\
&\rho_1 (\eta_{i+1},\eta_i')\leq \frac{\eta_i' - \eta_{i+1}}{\sqrt{1 - (\eta'_i)^2}}= \frac{(x_{i+1} - x_i')\epsilon}{\sqrt{x_i' (1 + \eta_i')}}
\end{align*}
Hence the condition on $\eta_i, \eta_i'$ is satisfied if the following hold
\begin{align*}
&x_i' \leq x_i + \sqrt{x_i (1 + \eta_i)}\\
&x_{i+1} \leq x_i' + \sqrt{x_i' (1 + \eta_i')}
\end{align*}
Therefore, we can construct $\eta_i,\eta_i'$ inductively by letting $x_1 = \frac{1}{2}, x_i' = x_i +\sqrt{x_i}, x_{i+1} = x_i + 2\sqrt{x_i}$.

By induction on $i$, it is easy to show that $x_i \geq
\frac{1}{3}i^2 + \frac{1}{6}i$. Let $k$ be the smallest positive
integer such that $(\frac{1}{3}k^2 + \frac{1}{6}k)\epsilon^2 \geq
1$, then $\eta_{k-1} > 0\geq \eta_k$. Thus
$\{\eta_1,\cdots,\eta_{k-1},0\}$ is an $\epsilon$-net for
$([0,1],\rho_1)$, and
\begin{equation*}
 k\leq [\sqrt{\frac{3}{\epsilon^2} + \frac{1}{16}} - \frac{1}{4}] + 1 \leq \frac{2}{\epsilon}
\end{equation*}
The last step uses the condition $0 < \epsilon \leq \frac{1}{3}$.

(iii): The construction of $\epsilon$-net in higher dimension is
based on two observations:  (a):  The restriction of $\rho_m$ on a
fixed radius is $\rho_1$, i.e. $\rho_m(sx,tx) = \rho_1(s,t),x\in
S^{m-1}, s,t\in [0,1]$;  (b):  The restriction of $\rho_m$ on the
sphere $S^{m-1} (r) = \{x\in B^m_2: |x| = r\}$ is the Euclidean
metric, i.e. $\rho_m(x,y) = |x-y|,x,y\in S^{m-1}(r)$.

Let's recall a basic fact about $\epsilon$-net of the sphere: For $0
< \epsilon < 2$, there exists an $\epsilon$-net for
$(S^{m-1},|\cdot|)$ with cardinality $\leq 2m(1 +
\frac{2}{\epsilon})^{m-1}$. Proof: Consider a maximal
$\epsilon$-separated subset $A$ of $S^{m-1}$, then $A$ is
automatically an $\epsilon$-net. The $\frac{\epsilon}{2}$-balls
centered at these points are disjoint and contained in
$(1+\frac{\epsilon}{2})B_2^m \backslash
(1-\frac{\epsilon}{2})B^m_2$, by volume counting we get
\begin{equation*}
|A|\cdot (\frac{\epsilon}{2})^m \leq (1+\frac{\epsilon}{2})^m - (1 - \frac{\epsilon}{2})^m \leq m(1+\frac{\epsilon}{2})^{m-1}\epsilon
\end{equation*}
Hence $|A|\leq 2m(1 + \frac{2}{\epsilon})^{m-1}$.

Let $\{x^1,\cdots, x^M\}$ be an $a\epsilon$-net for
$(S^{m-1},\rho_m)$, $M\leq 2m(1+\frac{2}{a\epsilon})^{m-1}, 0 < a <
1$. Since $0 < (1-a)\epsilon\leq \frac{1}{3}$, by (ii), we can find
an $\epsilon$-net $\{\eta_1,\cdots,\eta_N\}$ for $([0,1],\rho_1)$
with $N\leq \frac{2}{(1-a)\epsilon}$. Then
\begin{equation*}
 \mathscr{N} = \{\eta_i x^j: 1\leq i \leq N, 1\leq j\leq M\}
\end{equation*}
is an $\epsilon$-net for $(B^m_2,\rho_m)$, and $|\mathscr{N}|\leq
NM\leq \frac{4m}{(1-a)\epsilon}(1 + \frac{2}{a\epsilon})^{m-1}$.
Choosing $a = 1-\frac{1}{m}$ gives the desired bound.

To see $\mathscr{N}$ is an $\epsilon$-net, pick $x\in B^m_2$, we can
find $\eta_i$ with $\rho_1 (|x|,\eta_i)\leq (1-a)\epsilon$, and
$x^j$ with $|\frac{x}{|x|} - x^j|\leq a\epsilon$ (If $x=0$,replace
$\frac{x}{|x|}$ by any $u\in S^{m-1}$). Then
\begin{align*}
&\rho_m (x,\eta_i \frac{x}{|x|}) = \rho_1 (|x|,\eta_i)\leq (1-a)\epsilon\\
&\rho_m (\eta_i \frac{x}{|x|}, \eta_i x^j) = \eta_i |\frac{x}{|x|} - x^j|\leq a\epsilon
\end{align*}
By triangle inequality, we have $\rho_m (x,\eta_i x^j)\leq
(1-a)\epsilon + a\epsilon = \epsilon$.
\end{proof}

\begin{lemma}\label{Epsilon_Net_Construction_2} Assume $1\leq m < n$, for $u\in B^m_2$, define
\begin{equation*}
 \mathscr{E}(u) = \{x\in S^{n-1}: (x_1,\cdots,x_m) = u\}
\end{equation*}
Let $\mathscr{N}$ be an $\epsilon$-net for $(B^m_2,\rho_m)$ where
$\rho_m$ is defined in lemma $\ref{Epsilon_Net_Construction_1}$,
then $\mathscr{X} = \cup_{u\in\mathscr{N}}\mathscr{E}(u)$ is an
$\epsilon$-net for $(S^{n-1},|\cdot|)$. When $m=1$, if $\mathscr{N}$
is an $\epsilon$-net for $([0,1],\rho_1)$, then
$\cup_{u\in\mathscr{N}}\mathscr{E}(u)$ is an $\epsilon$-net for
$(S^{n-1}_{+},|\cdot|)$.
\end{lemma}

\begin{proof}
For $x\in S^{n-1}$, write $x = (x',x'')$, where $x'$ is the first
$m$ coordinates. Since $\mathscr{N}$ is an $\epsilon$-net for
$(B^m_2,\rho_m)$, we can find $u\in\mathscr{N}$ with
$\rho_m(x',u)\leq\epsilon$. Then $y = (u,\frac{\sqrt{1 -
|u|^2}}{\sqrt{1 - |x'|^2}}x'')\in\mathscr{E}(u)\subset\mathscr{X}$
(If $x''=0$,let $y = (u,0)$) is such that
\begin{equation*}
 |x - y| = \sqrt{|x' - u|^2 + |x'' - \frac{\sqrt{1 - |u|^2}}{\sqrt{1 - |x'|^2}}x''|^2} = \rho_m(x', u)\leq \epsilon
\end{equation*}

The proof for the $m=1,(S^{n-1}_{+},|\cdot|)$ case is similar.
\end{proof}

\begin{proposition}\emph{\cite{cis76}}\label{Gaussian_Concentration}
Let $X_1,\cdots,X_n$ be i.i.d. $\mathcal{N}(0,1)$ random variables.
$f:\mathbb{R}^n \rightarrow \mathbb{R}$ is Lipschitz. Then
\begin{equation*}
P(f(X) \geq E[f(X)] + t) \leq e^{\frac{-t^2}{2\|f\|_{Lip}^2}},\quad t\geq 0
\end{equation*}
\end{proposition}

\begin{proposition}\label{Chi_Square_Tail_Bound}
Let $X_1,\cdots,X_n$ be i.i.d. $\mathcal{N}(0,1)$ random variables,
$a_1,\cdots,a_n >0$, then for $t\geq 0$
\begin{align*}
&P(\sum_{i=1}^n a_i (X_i^2 - 1) \geq |a| t + \frac{1}{2}|a|_{\infty} t^2) \leq e^{-\frac{t^2}{4}}\\
&P(\sum_{i=1}^n a_i (X_i^2 - 1) \leq - |a| t)\leq e^{-\frac{t^2}{4}}
\end{align*}
\end{proposition}
The proof is based on Chernoff's exponential method, see \cite{lm00} page 1325.

\begin{proposition}\emph{\cite{mcdiarmid89} (McDiarmid's Inequality)}\label{McDiarmid_Inequality}
Let $X_i$ be an $\mathbb{S}_i$-valued random variable, $1\leq i\leq
n$, and assume $X_1,\cdots,X_n$ to be independent. Let $f:
\mathbb{S}_1\times\cdots\times\mathbb{S}_n\rightarrow\mathbb{R}$ be
Borel measurable. Suppose that there exist positive constants
$c_1,\cdots,c_n$ such that
\begin{equation*}
 |f(x) - f(x_1,\cdots,x_{i-1},x_i',x_{i+1},\cdots,x_n)|\leq c_i\quad \forall x, x_i'
\end{equation*}
Then for $t\geq 0$
\begin{equation*}
 P(f(X) \geq E[f(X)] + t) \leq e^{\frac{-2t^2}{\sum_{i=1}^n c_i^2}}
\end{equation*}
\end{proposition}

\begin{proposition}\emph{(Cauchy's Interlacing Law)}\label{Cauchy_Interlacing_Law}
Let $A\in\mathbb{R}^{n\times n}_{sym}$, deleting the $i_0$-th row
and $i_0$-th colume of $A$, we get a matrix $A_{i_0}$. Then the
eigenvalues of $A$ and $A_{i_0}$ satisfy
\begin{equation*}
\lambda_1 (A)\geq \lambda_1 (A_{i_0}) \geq \lambda_2 (A) \geq \lambda_2 (A_{i_0})\geq\cdots \geq \lambda_{n-1}(A_{i_0}) \geq \lambda_n (A)
\end{equation*}
\end{proposition}

\begin{proposition}\emph{\cite{ktt08}}\label{GOE_Convergence_Rate}
Let $G\in GOE(n,\frac{\sigma^2}{n})$, $F_n(x) = \frac{1}{n}|\{j:
\lambda_j (G)\leq x\}|$. Let $F(x)$ be the distribution function of
the semicircle law with density
$\frac{1}{2\pi\sigma^2}\sqrt{4\sigma^2 - x^2}1_{|x|\leq 2\sigma}dx$.
Then there exists a constant $C$ such that
\begin{equation*}
\sup_{x\in\mathbb{R}}|EF_n(x) - F(x)|\leq \frac{C}{n}
\end{equation*}
\end{proposition}

\begin{proposition}\emph{\cite{gt05}}\label{MP_Convergence_Rate}
Let $G\in\mathbb{R}^{p\times n}$ with entries being i.i.d.
$\mathcal{N}(0,1)$, $S_n = \frac{1}{n}GG^*$, $F_n
(x)=\frac{1}{p}|\{j: \lambda_j (S_n)\leq x\}|$. Let $F(x)$ be the
distribution function of the law with density
\begin{equation*}
p_c(x) = \begin{cases} \frac{\sqrt{((1+\sqrt{c})^2 -
x)(x - (1-\sqrt{c})^2)}}{2\pi c x} & \quad\text{if } (1-\sqrt{c})^2 \leq x \leq (1+\sqrt{c})^2,\\
0 & \quad\text{otherwise.}
\end{cases}
\end{equation*}
with $c=p/n$. Then there exists a constant $C$ such that
\begin{equation*}
\sup_{x\in\mathbb{R}}|EF_n(x) - F(x)|\leq C(\frac{1}{n}+\frac{1}{p})
\end{equation*}
\end{proposition}

\section*{Acknowledgements}
I am grateful to my Ph.D.~advisor, George C. Papanicolaou, for his
encouragement and help in preparing this manuscript.

\end{document}